\documentclass[12pt]{amsart}
\usepackage{amsmath}
\usepackage{amsthm}
\usepackage{tikz}
\usepackage{tikz-cd}
\usepackage{hyperref}
\usepackage[left=2.5cm,right=2.5cm,top=3cm,bottom=3cm]{geometry}
\usepackage{eucal}
\usepackage{palatino}
\usepackage{euler}
\usepackage{mathrsfs}
\usepackage{amssymb}
\usepackage{amscd}
\usepackage{latexsym}
\usepackage{epsfig}
\usepackage{graphicx}
\usepackage{amsfonts}

\usepackage{psfrag}
\usepackage{caption}
\usepackage{rotating}
\usepackage{mathtools}
\usepackage{fullpage}
\usepackage{etoolbox} 

\usepackage[normalem]{ulem}

\usepackage{pifont}
\usepackage{epsfig}

\usepackage{url}
\usepackage{cite}
\usepackage{dsfont}
\usepackage{array}

\newtheorem{theorem}{Theorem}[section]
\newtheorem{proposition}[theorem]{Proposition}
\newtheorem{corollary}[theorem]{Corollary}
\newtheorem{lemma}[theorem]{Lemma}

\theoremstyle{definition}
\newtheorem{definition}[theorem]{Definition}
\newtheorem{example}[theorem]{Example}
\newtheorem{remark}[theorem]{Remark}

\DeclareMathOperator{\GL}{GL}
\DeclareMathOperator{\SL}{SL}

\DeclareMathOperator{\Adj}{Ad}

\DeclareMathOperator{\adj}{ad}

\DeclareMathOperator{\Hom}{Hom}

\DeclareMathOperator{\Spec}{Spec}
\DeclareMathOperator{\trace}{tr}

\newtoggle{showColor}
\newtoggle{showNotes}
\toggletrue{showColor}
\togglefalse{showNotes}

\iftoggle{showNotes} {%
	\newcommand{\note}[1]{{\textcolor{red}{$\langle$#1$\rangle$}}} 
}{%
	\newcommand{\note}[1]{}
}

\title{Complex adjoint orbits in Lie theory and geometry}
\author{Peter Crooks}
\address{Institute of Differential Geometry, Gottfried Wilhelm Leibniz Universit\"{a}t Hannover, Welfengarten 1, 30167 Hannover, Germany}
\email{~~~peter.crooks@math.uni-hannover.de}

\begin{document}
\begin{abstract}
This expository article is an introduction to the adjoint orbits of complex semisimple groups, primarily in the algebro-geometric and Lie-theoretic contexts, and with a pronounced emphasis on the properties of semisimple and nilpotent orbits. It is intended to build a foundation for more specialized settings in which adjoint orbits feature prominently (ex. hyperk\"{a}hler geometry, Landau-Ginzburg models, and the theory of symplectic singularities). Also included are a few arguments and observations that, to the author's knowledge, have not yet appeared in the research literature.          
\end{abstract}


\maketitle

\tableofcontents

\section{Introduction}\label{Section: Introduction}
\subsection{Motivation and objectives}
In linear algebra, it is a classical exercise to determine whether two given matrices, $A,B\in\text{Mat}_{n\times n}(\mathbb{C})$, are conjugate, ie. whether $B=QAQ^{-1}$ for some invertible $Q\in\text{Mat}_{n\times n}(\mathbb{C})$. More generally, one can fix $A\in\text{Mat}_{n\times n}(\mathbb{C})$ and consider the set of all $B\in\text{Mat}_{n\times n}(\mathbb{C})$ that are conjugate to $A$. This set is precisely the conjugacy class of $A$, and, as we shall see, it is the paradigm example of a complex adjoint orbit. Despite this classical motivation, complex adjoint orbits are central to modern research in algebraic geometry \cite{Beauville,BeauvilleSingular,Fu,FuNamikawa,Hesselink,Kaledin,Kraft}, differential geometry \cite{BielawskiOnthe,BielawskiReducible,Biquard,Santa-Cruz,Kovalev,KronheimerInstantons,KronheimerSemisimple}, Lie theory \cite{Collingwood,Vogan,Sommers,Sommers2,Graham2}, and geometric representation theory \cite{Chriss,Jantzen,Springer,Borho}           
 
This article is an introduction to complex adjoint orbits in algebraic geometry, Lie theory, and, to a lesser extent, the other fields mentioned above. It is written with a view to meeting several key objectives. Firstly, it should be accessible to a fairly broad mathematical audience and require only introductory algebraic geometry and Lie theory as prerequisites. Secondly, it aims to build a clear, intuitive, and reasonably self-contained foundation for more advanced topics in which complex adjoint orbits play a role. Thirdly, it is designed to motivate and outline some of these more advanced topics (ex. hyperk\"{a}hler geometry, Landau-Ginzburg models, and symplectic singularities). Our final objective is to emphasize a few contextually appropriate arguments that, while possibly well-known to experts, do not seem to appear in the research literature. The arguments in question are the proofs of Proposition \ref{Proposition: Stabilizer structure}, Lemma \ref{Lemma: K-stabilizer}, Proposition \ref{Proposition: Diffeomorphism with the cotangent bundle}, Corollary \ref{Corollary: Equivariant homotopy equivalence}, Proposition \ref{Proposition: Intersection of opposite parabolics}, Theorem \ref{Theorem: Equivariant projective compactification}, and Proposition \ref{Proposition: Nilpotent orbits form a connected component}. While these are new proofs, at least to the author's knowledge, we emphasize that the underlying results themselves are well-known.                   

\subsection{Structure of the article} 
This article is organized as follows. Section \ref{Section: Preliminaries on algebraic group actions} addresses some of the relevant background on actions of complex algebraic groups, focusing largely on the structure of algebraic group orbits. It includes a number of examples, some of which are designed to give context and build intuition for later sections. Section \ref{Section: Preliminaries on Lie theory} subsequently discusses a few critical topics in Lie theory. Firstly, in the interest of clarity and consistency in later sections, \ref{Subsection: The basic objects} fixes various objects that arise when one studies representations of a complex semisimple group (ex. a maximal torus, a Borel subgroup, weights, roots, etc.). Secondly, \ref{Subsection: semisimple and nilpotent elements} and \ref{Subsection: The Jacobson-Morozov Theorem and sl2-triples} review the basics of semisimple and nilpotent elements in a complex semisimple Lie algebra.

Section \ref{Section: General adjoint orbits} specializes the previous sections to examine the adjoint orbits of a complex semisimple group. Adjoint orbits are defined in \ref{Subsection: Definitions and conventions}, and \ref{Subsection: Dimension and regularity} then introduces \textit{regular} adjoint orbits. Next, \ref{Subsection: Some first results} presents some results relating Jordan decompositions to the study of adjoint orbits. Section \ref{Subsection: The adjoint quotient} discusses adjoint orbits in the context of the \textit{adjoint quotient}, concluding with some of Kostant's foundational results (see Theorem \ref{Theorem: Kostant's theorem}). The topics become more manifestly geometric in \ref{Subsection: Geometric features}, which mentions the role of adjoint orbits in hyperk\"{a}hler and holomorphic symplectic geometry.

Section \ref{Section: Semisimple orbits} is devoted to semisimple adjoint orbits and their features. These orbits are introduced in \ref{Subsection: Definitions and characterizations}, where it is also shown that semisimple orbits are exactly the closed orbits of the adjoint action (see Theorem \ref{Theorem: Semisimple equivalent to closed}). Section \ref{Subsection: Stabilizer descriptions} develops a uniform description of a semisimple element's stabilizer under the adjoint action of a complex semisimple group (see Proposition \ref{Proposition: Stabilizer structure} and Corollary \ref{Corollary: Stabilizers of semisimple elements are reductive}). The discussion becomes more esoteric in Sections \ref{Subsection: Fibrations over partial flag varieties} and \ref{Subsection: Equivariant projective compactifications}. The former shows semisimple orbits to be explicitly diffeomorphic to cotangent bundles of partial flag varieties $G/P$ (see Proposition \ref{Proposition: Diffeomorphism with the cotangent bundle}), which is consistent with results in hyperk\"{a}hler geometry (see Remark \ref{Remark: Connection to hyperkahler geometry}). Section \ref{Subsection: Equivariant projective compactifications} reviews a procedure for compactifying semisimple orbits (see Theorem \ref{Theorem: Equivariant projective compactification}), and subsequently gives some context from the homological mirror symmetry program (see Remark \ref{Remark: Connection to mirror symmetry}).

Section \ref{Section: Nilpotent orbits} is exclusively concerned with nilpotent adjoint orbits and their properties. It begins with \ref{Subsection: Definitions and first results}, in which nilpotent orbits are defined and then characterized in several equivalent ways. Also included is a proof that there are only finitely many nilpotent orbits of a given complex semisimple group (see Theorem \ref{Theorem: Finitely many nilpotent orbits}). Section \ref{Subsection: The closure order on nilpotent orbits} studies the set of nilpotent orbits as a poset, partially ordered by the closure order. In particular, it recalls Gerstenhaber and Hesselink's classical description of this poset in Lie type $A$ (see Example \ref{Example: Nilpotent orbits and the dominance order}). Section \ref{Subsection: The regular nilpotent orbit} shows there to be a unique maximal element in the poset of nilpotent orbits (see Proposition \ref{Proposition: Unique maximal}), called the \textit{regular nilpotent orbit}. We also recall Kostant's construction of a standard representative for this orbit (see Proposition \ref{Proposition: Nilpotent orbit representative}). Next, \ref{Subsection: The minimal nilpotent orbit} introduces $\mathcal{O}_{\text{min}}$, the minimal (non-zero) nilpotent orbit of a complex simple group. This leads to a discussion of nilpotent orbit projectivizations in \ref{Subsection: Orbit projectivizations}, where $\mathbb{P}(\mathcal{O}_{\text{min}})$ and its properties are discussed in detail. Some emphasis is placed on the role of $\mathbb{P}(\mathcal{O}_{\text{min}})$ in quaternionic-K\"{a}hler geometry. Lastly, \ref{Subsection: Orbit closures and singularities} outlines how nilpotent orbit closures can arise in geometric representation theory, and in the study of symplectic singularities and resolutions.                                    

\subsection{Acknowledgements}
The author is grateful to Roger Bielawski, Steven Rayan, and Markus R\"{o}ser for many fruitful discussions, and to the Institute of Differential Geometry at Leibniz Universit\"{a}t Hannover for its hospitality.

\subsection{General conventions}

While many of our conventions will be declared in Sections \ref{Section: Preliminaries on algebraic group actions} and \ref{Section: Preliminaries on Lie theory}, the following are some things we can establish in advance.
\begin{itemize}
\item We will discuss several objects that, strictly speaking, make sense only when a base field has been specified (ex. vector spaces, Lie algebras, representations, algebraic varieties, etc.). Unless we indicate otherwise, we will always take this field to be $\mathbb{C}$.
\item If $X$ is a set with an equivalence relation $\sim$, then $[x]\in X/{\sim}$ will denote the equivalence class of $x\in X$. The equivalence relations of interest to us will usually come from the (left) action of a group $G$ on a set $X$, in which $x_1\sim x_2$ if and only if $x_2=g\cdot x_1$ for some $g\in G$. In such cases, we will denote the quotient set $X/{\sim}$ by $X/G$.
\item Many of the spaces we will consider can be viewed both as algebraic varieties with their Zariski topologies, and as manifolds with their analytic topologies. Accordingly, if one of our statements has topological content, the reader should by default understand the underlying topology to be the Zariski topology. Only when a space does not have a clear algebraic variety structure (ex. a compact Lie group or an orbit thereof), or when manifold-theoretic terms are being used (ex. diffeomorphism, holomorphic, homotopy-equivalence, equivariant cohomology) should we be understood as referring to the analytic topology.            
\end{itemize} 

\section{Preliminaries on algebraic group actions}\label{Section: Preliminaries on algebraic group actions}
In what follows, we review some pertinent concepts from the theory of group actions on algebraic varieties. This review is designed to serve two principal purposes: to clearly establish some of our major conventions, and to give motivation/context for what is to come. However, it is not intended to be a comprehensive overview of algebraic group actions. For this, we refer the reader to \cite[Chapt. II--IV]{Humphreys}.    
\subsection{Definitions and examples}
Let $G$ be a connected linear algebraic group with Lie algebra $\mathfrak{g}$ and adjoint representation $\Adj:G\rightarrow\GL(\mathfrak{g})$. By $G$-\textit{variety}, we shall mean an algebraic variety $X$ equipped with an algebraic $G$-action, namely a variety morphism\\ $G\times X\rightarrow X$, $(g,x)\mapsto g\cdot x$, satisfying the usual properties of a left $G$-action:
\begin{itemize}
	\item[(i)] $e\cdot x=x$ for all $x\in X$, where $e\in G$ is the identity element,
	\item[(ii)] $(g_1g_2)\cdot x=g_1\cdot (g_2\cdot x)$ for all $g_1,g_2\in G$ and $x\in X$.
\end{itemize}
A morphism of $G$-varieties is then appropriately defined to be a $G$-equivariant morphism of algebraic varieties, ie. a variety morphism $\varphi:X\rightarrow Y$ satisfying $\varphi(g\cdot x)=g\cdot\varphi(x)$ for all $g\in G$ and $x\in X$. This induces the definition of a $G$-variety isomorphism, which is precisely a $G$-equivariant isomorphism of algebraic varieties.        

The literature contains many examples of $G$-varieties and related ideas, some of which we mention below.

\begin{example}\label{Example: Left-multiplication} Let $G=\GL_n(\mathbb{C})$ act by left multiplication on the vector space of $n\times n$ matrices $\text{Mat}_{n\times n}(\mathbb{C})$, so that $g\cdot A=gA$, $g\in GL_n(\mathbb{C}),A\in\text{Mat}_{n\times n}(\mathbb{C}).$ In this way, $\text{Mat}_{n\times n}(\mathbb{C})$ is a $\GL_n(\mathbb{C})$-variety.
\end{example}	

\begin{example}\label{Example: Conjugation example}Alternatively, one has the conjugation action of $\GL_n(\mathbb{C})$ on $\text{Mat}_{n\times n}(\mathbb{C})$, ie. $g\cdot A=gAg^{-1}$, $g\in GL_n(\mathbb{C}),A\in\text{Mat}_{n\times n}(\mathbb{C}).$ This, too, renders $\text{Mat}_{n\times n}(\mathbb{C})$ a $\GL_n(\mathbb{C})$-variety.
\end{example}

\begin{example}\label{Example: partial flag variety in type A}
By a \textit{flag in } $\mathbb{C}^n$, we shall mean a sequence $$V_{\bullet}=(\{0\}\subsetneq V_1\subsetneq V_2\subsetneq\ldots\subsetneq V_k\subsetneq\mathbb{C}^n)$$ of subspaces of $\mathbb{C}^n$. Now, suppose that $d_{\bullet}=(d_1,d_2,\ldots,d_k)$ is a strictly increasing sequence of integers with $1\leq d_i\leq n-1$ for all $i=1,\ldots,k$. We denote by $\text{Flag}(d_{\bullet},\mathbb{C}^n)$ the set of all flags in $\mathbb{C}^n$ whose constituent subspaces have dimensions $d_1,d_2,\ldots,d_k$, namely 
$$\text{Flag}(d_{\bullet};\mathbb{C}^n)=\{V_{\bullet}:\dim(V_i)=d_i,\text{ }i=1,\ldots,k\}.$$ This is a smooth projective variety and it carries a $\GL_n(\mathbb{C})$-variety structure in which $\GL_n(\mathbb{C})$ acts via
$$g\cdot V_{\bullet}=(\{0\}\subsetneq gV_1\subsetneq gV_2\subsetneq\ldots\subsetneq gV_k\subsetneq\mathbb{C}^n),\quad g\in\GL_n(\mathbb{C}),V_{\bullet}\in\text{Flag}(d_{\bullet},\mathbb{C}^n).$$
\end{example} 

\begin{example}
Let $X$ be a $G$-variety and $H\subseteq G$ a closed subgroup. One can restrict the $G$-action to $H$ in the sense that $H\times X\rightarrow X$, $(h,x)\mapsto h\cdot x$, $h\in H,x\in X$ defines an $H$-variety structure on $X$.
\end{example}

\begin{example}\label{Example: restriction to an invariant subvariety}
Let $X$ be a $G$-variety and $Y\subseteq X$ a locally closed subvariety.\footnote{Recall that a subset of $X$ is defined to be locally closed if it is expressible as the intersection of an open and a closed subset of $X$. Equivalently, a subset is locally closed if and only if the subset is open in its closure.} If $g\cdot y\in Y$ for all $g\in G$ and $y\in Y$, we will refer to $Y$ as being $G$-\textit{invariant} (or simply \textit{invariant} if the action is clear from context). In this case, the map $G\times Y\rightarrow Y$, $(g,y)\mapsto g\cdot y$, is a $G$-variety structure on $Y$.
\end{example}

\begin{example}\label{Example: Diagonal action}
Let $X$ and $Y$ be $G$-varieties. The product $X\times Y$ then carries a canonical $G$-variety structure, defined by the ``diagonal'' $G$-action $$g\cdot(x,y)=(g\cdot x,g\cdot y),\quad g\in G,\text{ }(x,y)\in X\times Y.$$ 
\end{example}
	
\begin{example}\label{Example: action from representation}
Let $V$ be a $G$-\textit{representation}, meaning for us that $V$ is a finite-dimensional vector space together with a morphism of algebraic groups $\rho:G\rightarrow\GL(V)$ (ex. $V=\mathfrak{g}$ and $\rho=\Adj$). In this case, the map $G\times V\rightarrow V$, $(g,v)\mapsto\rho(g)(v)$, gives $V$ the structure of a $G$-variety. Note that Examples \ref{Example: Left-multiplication} and \ref{Example: Conjugation example} feature $G$-varieties induced by representations.
\end{example}

\begin{example}\label{Example: Isotropy representations}
As a counterpart to the last example, one can use $G$-varieties to induce representations. To this end, let $X$ be a $G$-variety and for each $g\in G$ let $\phi_g:X\rightarrow X$ be the automorphism defined by $\phi_g(x)=g\cdot x$. Now suppose $H\subseteq G$ is a closed subgroup and $x\in X$ is a point satisfying $h\cdot x=x$ (ie. $\phi_h(x)=x$) for all $h\in H$. It follows that the differential of $\phi_h$ at $x$ is a vector space automorphism $d_x(\phi_h):T_xX\rightarrow T_xX$ of the tangent space $T_xX$. Furthermore, one can verify that $H\rightarrow\GL(T_xX)$, $h\mapsto d_x(\phi_h)$, is an $H$-representation.
\end{example}

\begin{example}\label{Example: Descent to projective space} Let $\rho:G\rightarrow\GL(V)$ be a $G$-representation, and let $\mathbb{P}(V)=(V\setminus\{0\})/\mathbb{C}^*$ denote the projectivization of $V$. The formula $g\cdot[v]=[\rho(g)(v)]$ gives a well-defined action of $G$ on $\mathbb{P}(V)$, rendering the latter a $G$-variety.
\end{example}

\begin{example}\label{Example: Induced action on functions}
Let $X$ be an affine $G$-variety and $\mathbb{C}[X]$ the algebra of regular functions on X. This algebra carries a canonical action of $G$ by algebra automorphisms, in which $g\in G$ acts on $f\in\mathbb{C}[X]$ by $(g\cdot f)(x):=f(g^{-1}\cdot x)$, $x\in X$.
\end{example}

\begin{example}\label{Example: Quotient by a closed subgroup} Let $H$ be a closed subgroup of $G$. The set of left $H$-cosets $X=G/H=\{gH:g\in G\}$ is naturally an algebraic variety. Furthermore, when $G$ acts by left multiplication, $G/H$ is a $G$-variety.   
\end{example}

A few remarks are in order. Firstly, one may summarize Examples \ref{Example: action from representation} and \ref{Example: Descent to projective space} together as the statement that a $G$-representation $V$ induces canonical $G$-variety structures on both $V$ and $\mathbb{P}(V)$. We will always understand $V$ and $\mathbb{P}(V)$ to be $G$-varieties in exactly this way when $V$ is a $G$-representation. Secondly, we shall always view a variety $G/H$ as carrying the left-multiplicative action of $G$.

\subsection{Generalities on orbits}\label{Subsection: Generalities on orbits}
Let $X$ be a $G$-variety. One may associate to each point $x\in X$ the subset \begin{equation}\label{Equation: Orbit definition}
\mathcal{O}_G(x):=\{g\cdot x:g\in G\},
\end{equation}
called the $G$-\textit{orbit of} $x$. Should the underlying group $G$ be clear from context, we will sometimes suppress it and write $\mathcal{O}(x)$ in place of $\mathcal{O}_G(x)$. Now, let us call $\mathcal{O}\subseteq X$ a $G$-\textit{orbit} if $\mathcal{O}=\mathcal{O}(x)$ for some $x\in X$. In this case, $\overline{\mathcal{O}}$ shall denote the closure of $\mathcal{O}$ in $X$. The orbit $\mathcal{O}$ is open in $\overline{\mathcal{O}}$ (ie. is locally closed), so that $\mathcal{O}$ inherits from $X$ the structure of a subvariety. Furthermore, $\mathcal{O}$ is known to be a smooth subvariety of $X$. 	

\begin{example}Let $\GL_n(\mathbb{C})$ act by left multiplication on $\text{Mat}_{n\times n}(\mathbb{C})$. If $A,B\in\text{Mat}_{n\times n}(\mathbb{C})$, then $B=gA$ for some $g\in\GL_n(\mathbb{C})$ if and only if $A$ and $B$ are equivalent under row operations, which holds if and only if their reduced echelon forms coincide. It follows that the $\GL_n(\mathbb{C})$-orbits are parametrized by the set of $n\times n$ matrices in reduced echelon form. In particular, the orbit of matrices with reduced echelon form $I_n$ is $\mathcal{O}(I_n)=\GL_n(\mathbb{C})\subseteq\text{Mat}_{n\times n}(\mathbb{C})$. This orbit is open and dense. 
\end{example}

\begin{example}\label{Example: Conjugation}	 
Let $\GL_n(\mathbb{C})$ act by conjugation on $\text{Mat}_{n\times n}(\mathbb{C})$, so that the $\GL_n(\mathbb{C})$-orbits are precisely the conjugacy classes of matrices in $\text{Mat}_{n\times n}(\mathbb{C})$. These classes are indexed by the equivalence classes of $n\times n$ matrices in Jordan canonical form, where two such matrices are equivalent if they agree up to a reordering of Jordan blocks along the diagonal. To build some algebro-geometric intuition for these orbits, we specialize to three cases. 
\begin{itemize}
\item Suppose that $A\in\text{Mat}_{n\times n}(\mathbb{C})$ is diagonal with pairwise distinct eigenvalues\\ $\lambda_1,\ldots,\lambda_n\in\mathbb{C}$. If $B\in\text{Mat}_{n\times n}(\mathbb{C})$, then $B\in\mathcal{O}(A)$ if and only if the characteristic polynomial of $B$ has the form $$\det(\lambda I_n-B)=(\lambda-\lambda_1)(\lambda-\lambda_2)\cdots(\lambda-\lambda_n).$$ Equating corresponding coefficients of $\lambda^k$ for each $k\in\{0,1,\ldots,n-1\}$, we obtain $n$ polynomial equations in the entries of $B$. From this, it follows that $\mathcal{O}(A)$ is a closed (hence affine) subvariety of $\text{Mat}_{n\times n}(\mathbb{C})$. 
\item Let $N_1\in\text{Mat}_{n\times n}(\mathbb{C})$ be the matrix consisting of a single nilpotent Jordan block. If $B\in\text{Mat}_{n\times n}(\mathbb{C})$, then $B\in\mathcal{O}(N_1)$ if and only $B$ has a minimal polynomial of $\lambda^n$, ie.
$$\mathcal{O}(N_1)=\{B\in\text{Mat}_{n\times n}(\mathbb{C}):B^n=0,\text{ }B^{n-1}\neq 0\}.$$ The conditions $B^n=0$ and $B^{n-1}\neq 0$ define closed and open subsets of $\text{Mat}_{n\times n}(\mathbb{C})$, respectively, and $\mathcal{O}(N_1)$ is the intersection of these subsets. We thus see explicitly that $\mathcal{O}(N_1)$ is locally closed. Also, one can show that $\overline{\mathcal{O}(N_1)}=\{B\in\text{Mat}_{n\times n}(\mathbb{C}):B^n=0\}$, the locus of all nilpotent matrices.
\item Let $N_2\in\text{Mat}_{n\times n}(\mathbb{C})$ be a nilpotent matrix in Jordan canonical form having a single $2\times 2$ Jordan block and all remaining blocks of dimension $1\times 1$. To describe $\mathcal{O}(N_2)$, note that the Jordan canonical form of a nilpotent $B\in\text{Mat}_{n\times n}(\mathbb{C})$ has only $1\times 1$ and $2\times 2$ blocks if and only if $B^2=0$. To ensure the existence of exactly one $2\times 2$ block, one must impose the condition $\text{rank}(B)=1$. Equivalently, one requires that $B\neq 0$ and that all $2\times 2$ minors of $B$ vanish. It follows that \begin{align*}\mathcal{O}(N_2) & =\{B\in\text{Mat}_{n\times n}(\mathbb{C}):B^2=0,\text{ rank}(B)=1\}\\ & = \{B\in\text{Mat}_{n\times n}(\mathbb{C}):B^2=0,\text{ every } 2\times 2\text{ minor of }B=0,\text{ }B\neq 0\}.
\end{align*}
As with $\mathcal{O}(N_1)$, we explicitly see that $\mathcal{O}(N_2)$ is the intersection of a closed subset of $\text{Mat}_{n\times n}(\mathbb{C})$ (defined by $B^2=0$ and the vanishing of the $2\times 2$ minors) with an open subset (defined by $B\neq 0$). Furthermore, it is not difficult to see that $0\in\overline{\mathcal{O}(N_2)}$, and that $\mathcal{O}(N_2)\cup\{0\}$ is the closed subset mentioned in the previous sentence. We conclude that $\overline{\mathcal{O}(N_2)}=\mathcal{O}(N_2)\cup\{0\}$.      
\end{itemize}     
\end{example}

\begin{example}\label{Example: second partial flag variety in type A}
Refer to Example \ref{Example: partial flag variety in type A}, in which $\GL_n(\mathbb{C})$ acts on $\text{Flag}(d_{\bullet};\mathbb{C}^n)$. One can verify that this action is transitive, so that $\text{Flag}(d_{\bullet};\mathbb{C}^n)$ is the unique $\GL_n(\mathbb{C})$-orbit. However, let us restrict the $\GL_n(\mathbb{C})$-action to the Borel subgroup $B$ of upper-triangular matrices. To describe the $B$-orbits, let $w\in S_n$ be a permutation and consider the point
$$V_{\bullet}^w:=(\{0\}\subsetneq V_{1}^w\subsetneq V_2^w\subsetneq\ldots\subsetneq V_{k}^w\subsetneq\mathbb{C}^n)\in \text{Flag}(d_{\bullet};\mathbb{C}^n),$$ where $V_i^w:=\text{span}\{e_{w(1)},e_{w(2)},\ldots,e_{w(d_i)}\}$, $i\in\{1,2,\ldots,k\}$. It turns out that $w\mapsto\mathcal{O}_B(V_{\bullet}^w)$ is a surjection from $S_n$ to the collection of $B$-orbits. Furthermore, $\mathcal{O}_B(V_{\bullet}^w)$ is isomorphic to affine space and called a \textit{Schubert cell}, alluding to the fact that the $B$-orbits constitute a cell decomposition of $\text{Flag}(d_{\bullet};\mathbb{C}^n)$. The $B$-orbit closures are the well-studied \textit{Schubert varieties}. For further details in this direction, we refer the reader to \cite[Chapt. 10]{Fulton}.   
\end{example}

\subsection{Orbits as quotient varieties}\label{Subsection: Orbits as quotient varieties}
Let $X$ be a $G$-variety and $\mathcal{O}\subseteq X$ an orbit. Since $\mathcal{O}$ is locally closed and $G$-invariant, it inherits from $X$ the structure of a $G$-variety (cf. Example \ref{Example: restriction to an invariant subvariety}). To better understand this structure, fix a point $x\in\mathcal{O}$ and consider its $G$-stabilizer
\begin{equation}\label{Equation: Stabilizer definition}
C_G(x):=\{g\in G:g\cdot x=x\}.
\end{equation}
This is a closed subgroup of $G$ and one has a well-defined $G$-variety isomorphism
\begin{equation}\label{Equation: Orbit-stabilizer isomorphism}
G/C_G(x)\xrightarrow{\cong}\mathcal{O},\quad [g]\mapsto g\cdot x.
\end{equation}
Note that \eqref{Equation: Orbit-stabilizer isomorphism} is just the Orbit-Stabilizer Theorem in an algebro-geometric context.

One can use \eqref{Equation: Orbit-stabilizer isomorphism} to yield a convenient description of the tangent space $T_x\mathcal{O}$. Indeed, since \eqref{Equation: Orbit-stabilizer isomorphism} sends the identity coset $[e]\in G/C_G(x)$ to $x$, it determines an isomorphism of tangent spaces, 
\begin{equation}\label{Equation: Vector space isomorphism}
T_{[e]}(G/C_G(x))\xrightarrow{\cong} T_x\mathcal{O}.
\end{equation}
Also, the differential of $G\rightarrow G/C_G(x)$ at $e\in G$ is surjective with kernel the Lie algebra of $C_G(x)$. Let us denote this Lie algebra by $C_{\mathfrak{g}}(x)$, so that we have $T_{[e]}(G/C_G(x))\cong\mathfrak{g}/C_{\mathfrak{g}}(x)$ as vector spaces. The isomorphism \eqref{Equation: Vector space isomorphism} then takes the form   
\begin{equation}\label{Equation: Tangent space description}
T_x\mathcal{O}\cong\mathfrak{g}/C_{\mathfrak{g}}(x).
\end{equation}

It turns out that \eqref{Equation: Tangent space description} is more than an isomorphism of vector spaces. To see this, note that $T_x\mathcal{O}$ is a $C_G(x)$-representation by virtue of $x$ being fixed by $C_G(x)$ (see Example \ref{Example: Isotropy representations}). At the same time, $\mathfrak{g}$ carries the adjoint representation of $G$, and one can restrict this to a representation of $C_G(x)$. One notes that $C_{\mathfrak{g}}(x)$ is a $C_G(x)$-invariant subspace, making the quotient $\mathfrak{g}/C_{\mathfrak{g}}(x)$ a $C_G(x)$-representation. With this discussion in mind, one can verify that \eqref{Equation: Tangent space description} is actually an isomorphism of $C_G(x)$-representations.

\begin{example}\label{Example: Stabilizer in a flag variety}
	Recall the action of $\GL_n(\mathbb{C})$ on $\text{Flags}(d_{\bullet};\mathbb{C}^n)$ from Example \ref{Example: partial flag variety in type A}. Following the notation of Example \ref{Example: second partial flag variety in type A}, the $\GL_n(\mathbb{C})$-stabilizer of the flag $V^e_{\bullet}$ is
	$$C_{\GL_n(\mathbb{C})}(V^e_{\bullet})=\{(A_{ij})\in\GL_n(\mathbb{C}):A_{ij}=0\text{ for all }r\in\{1,\ldots,k\}, i>d_r, j\in\{d_{r-1},\ldots,d_r\}\}.$$ More simply, a matrix $A\in\GL_n(\mathbb{C})$ is in $C_{\GL_n(\mathbb{C})}(V^e_{\bullet})$ if and only if it is block upper-triangular with blocks of dimensions $d_1\times d_1, (d_2-d_1)\times (d_2-d_1),\ldots,(d_k-d_{k-1})\times (d_k-d_{k-1}),(n-d_k)\times (n-d_k)$ along the diagonal (read from top to bottom). Now, since the action of $\GL_n(\mathbb{C})$ on $\text{Flags}(d_{\bullet};\mathbb{C}^n)$ can be shown to be transitive, \eqref{Equation: Orbit-stabilizer isomorphism} gives an isomorphism
	$$\GL_n(\mathbb{C})/C_{\GL_n(\mathbb{C})}(V^e_{\bullet})\cong \text{Flags}(d_{\bullet};\mathbb{C}^n).$$    
\end{example}

\begin{example}
	Recall Example \ref{Example: Conjugation}, and let $A,N_1\in\text{Mat}_{n\times n}(\mathbb{C})$ be as introduced there. One can verify that $C_{\GL_n(\mathbb{C})}(A)$ is the maximal torus of diagonal matrices in $\GL_n(\mathbb{C})$, while $C_{\GL_n(\mathbb{C})}(N_1)$ can be seen to have the following description:
	$$C_{\GL_n(\mathbb{C})}(N_1)=\{(B_{ij})\in\GL_n(\mathbb{C}):B_{ij}=0\text{ for }i>j, \text{ }B_{ij}=B_{(i+1)(j+1)}\text{ for }1\leq i\leq j\leq n-1\}.$$
	In other words, a matrix $B\in\GL_n(\mathbb{C})$ is in  $C_{\GL_n(\mathbb{C})}(N_1)$ if and only if $B$ is upper-triangular and any two entires of $B$ lying on the same diagonal must be equal. Now, let $U\subseteq C_{\GL_n(\mathbb{C})}(N_1)$ be the unipotent subgroup of matrices having only $1$ along the main diagonal. This subgroup is complementary to the centre $Z(\GL_n(\mathbb{C}))$, and one has an internal direct product decomposition
	$C_{\GL_n(\mathbb{C})}(N_1)=Z(\GL_n(\mathbb{C}))\times U$. 
\end{example}

\subsection{The closure order}\label{Subsection: The closure order}
The set of orbits in a $G$-variety $X$ turns out to have a canonical partial order, based fundamentally on the following proposition. 

\begin{proposition}\label{Proposition: orbit boundary}
If $\mathcal{O}\subseteq X$ is a $G$-orbit, then $\overline{\mathcal{O}}$ is a union of $\mathcal{O}$ and orbits having dimensions strictly less than $\dim(\mathcal{O})$.	
\end{proposition}
	
\begin{proof}
Choose $x\in\mathcal{O}$ and note that $\mathcal{O}$ is the image of the morphism $G\rightarrow X$, $g\mapsto g\cdot x$. Since $G$ is irreducible, so too is the image $\mathcal{O}$. It follows that $\overline{\mathcal{O}}$ is also irreducible. Appealing to the fact that $\mathcal{O}$ is locally closed (ie. open in $\overline{\mathcal{O}}$), we see that $\overline{\mathcal{O}}\setminus\mathcal{O}$ is a closed subvariety of $\overline{\mathcal{O}}$, each of whose irreducible components then necessarily has dimension strictly less than $\dim(\overline{\mathcal{O}})=\dim(\mathcal{O})$ (see \cite[Sect. 3.2]{Humphreys}). Of course, as $\overline{\mathcal{O}}\setminus\mathcal{O}$ is $G$-invariant, it is a union of $G$-orbits. Each of these orbits is also irreducible and therefore has dimension at most that of an irreducible component of $\overline{\mathcal{O}}\setminus\mathcal{O}$ to which it belongs. Our previous discussion implies that the dimension of such an orbit is strictly less than $\dim(\mathcal{O})$.   
\end{proof}

Proposition 2 implies that each $G$-orbit $\mathcal{O}$ is the unique orbit in $\overline{\mathcal{O}}$ having dimension $\dim(\mathcal{O})$.
With this in mind, suppose that $\mathcal{O}_1,\mathcal{O}_2\subseteq X$ are $G$-orbits satisfying $\mathcal{O}_1\subseteq\overline{\mathcal{O}_2}$ and $\mathcal{O}_2\subseteq\overline{\mathcal{O}_1}$. These inclusions show $\mathcal{O}_1$ and $\mathcal{O}_2$ to have the same dimension. Since both orbits lie in $\overline{\mathcal{O}_1}$, it follows that $\mathcal{O}_1=\mathcal{O}_2$ must hold. It is thus not difficult to see that 
\begin{equation}\label{Equation: closure order}
\mathcal{O}_1\leq\mathcal{O}_2\Longleftrightarrow\mathcal{O}_1\subseteq\overline{\mathcal{O}_2}
\end{equation}
defines a partial order on the set of $G$-orbits in $X$, which we shall call the \textit{closure order}.  

\begin{example}
Recall the setup of Example \ref{Example: Conjugation}, in which $\GL_n(\mathbb{C})$ acts by conjugation on $\text{Mat}_{n\times n}(\mathbb{C})$ and $A,N_1,N_2\in\text{Mat}_{n\times n}(\mathbb{C})$ are three fixed matrices. Having shown $\mathcal{O}(A)$ to be closed in $\text{Mat}_{n\times n}(\mathbb{C})$, we see that $\mathcal{O}(A)$ is a minimal element in the closure order. Secondly, remembering that $\overline{\mathcal{O}(N_1)}$ is the locus of all nilpotent $n\times n$ matrices, the $G$-orbits in $\overline{\mathcal{O}(N_1)}$ are the conjugacy classes of nilpotent matrices. These conjugacy classes are therefore precisely the orbits $\mathcal{O}\subseteq\text{Mat}_{n\times n}(\mathbb{C})$ satisfying $\mathcal{O}\leq\mathcal{O}(N_1)$. Finally, having seen in Example \ref{Example: Conjugation} that  $\overline{\mathcal{O}(N_2)}=\mathcal{O}(N_2)\cup\{0\}$, we conclude that $\{0\}$ and $\mathcal{O}(N_2)$ are the unique orbits $\mathcal{O}$ satisfying $\mathcal{O}\leq\mathcal{O}(N_2)$.
\end{example}

\begin{example}
Refer to Examples \ref{Example: partial flag variety in type A} and \ref{Example: second partial flag variety in type A}, where $\GL_n(\mathbb{C})$ and its subgroup $B$ act on $\text{Flag}(d_{\bullet};\mathbb{C}^n)$. We will assume that $d_{\bullet}=(1,2,\ldots,n-1)$, in which case $w\mapsto\mathcal{O}_B(V_{\bullet}^w)$ is a bijection from $S_n$ to the set of $B$-orbits. The latter set has the closure order, and there is a unique partial order on $S_n$ for which our bijection is a poset isomorphism. One thereby obtains the well-studied \textit{Bruhat order} on $S_n$, which we now describe. Given $i\in\{1,\ldots,n-1\}$, let $s_i\in S_n$ be the transposition interchanging $i$ and $i+1$. The $n-1$ transpositions $s_i$ generate $S_n$, so that each $w\in S_n$ has an expression $w=s_{i_1}s_{i_2}\cdots s_{i_k}$. One calls such an expression \textit{reduced} if its length $k$ is minimal. Now given $v,w\in S_n$, the Bruhat order is defined as follows: $v\leq w$ if whenever $w=s_{i_1}s_{i_2}\cdots s_{i_k}$ is a reduced expression for $w$, it is possible to delete some of the $s_{i_j}$ so that the resulting expression is a reduced expression for $v$.     
\end{example}

\section{Preliminaries on Lie theory}\label{Section: Preliminaries on Lie theory}
The following is a brief review of Lie-theoretic topics central to this article, with a strong emphasis on the structure of semisimple algebraic groups and their Lie algebras. Like the previous section, this review is more concerned with establishing clear conventions than being comprehensive. References for the relevant Lie theory include \cite[Chapt. III--VIII]{Humphreys} and \cite[Chapt. 20--29]{Tauvel}.   

\subsection{The basic objects}\label{Subsection: The basic objects}
In what follows, we introduce several Lie-theoretic objects to be understood as fixed for the rest of this article. To begin, let $G$ be a connected, simply-connected semisimple linear algebraic group having Lie algebra $\mathfrak{g}$ and exponential map $\exp:\mathfrak{g}\rightarrow G$. Let $$\Adj:G\rightarrow\GL(\mathfrak{g}),\quad g\mapsto\Adj_g,\text{ }g\in G$$ and $$\adj:\mathfrak{g}\rightarrow\mathfrak{gl}(\mathfrak{g}),\quad x\mapsto\adj_x,\text{ }x\in\mathfrak{g}$$ denote the adjoint representations of $G$ and $\mathfrak{g}$, respectively. Recall that $\Adj_g$ is a Lie algebra automorphism of $\mathfrak{g}$ for all $g\in G$, a fact we will use repeatedly. Note also that $\adj_x(y)=[x,y]$ for all $x,y\in\mathfrak{g}$, from which it follows that the kernel of $\adj$ is the centre of $\mathfrak{g}$. This centre is trivial by virtue of the fact that $\mathfrak{g}$ is semisimple. In other words, $\adj$ is injective. 

Now, suppose that $T\subseteq B\subseteq G$, with $T$ a maximal torus of $G$ and $B$ a Borel subgroup of $G$. One then has inclusions $\mathfrak{t}\subseteq\mathfrak{b}\subseteq\mathfrak{g}$, where $\mathfrak{t}$ and $\mathfrak{b}$ are the Lie algebras of $T$ and $B$, respectively. Let $X^*(T)$ be the additive group of weights (ie. algebraic group morphisms $T\rightarrow\mathbb{C}^*$), with the sum of $\lambda_1,\lambda_2\in X^*(T)$ defined by $(\lambda_1+\lambda_2)(t)=\lambda_1(t)\lambda_2(t)$ for all $t\in T$. The map sending $\lambda\in X^*(T)$ to its differential at the identity $e\in T$ includes $X^*(T)$ into $\mathfrak{t}^*$ as an additive subgroup. Where appropriate, we will use this inclusion to regard a weight as belonging to $\mathfrak{t}^*$. 

Given a weight $\alpha\in X^*(T)$, let us set $$\mathfrak{g}_{\alpha}:=\{x\in\mathfrak{g}:\Adj_t(x)=\alpha(t)x\text{ for all }t\in T\}.$$ If we regard $\alpha$ as belonging to $\mathfrak{t}^*$, then $\mathfrak{g}_{\alpha}$ admits the following alternative description:
$$\mathfrak{g}_{\alpha}=\{x\in\mathfrak{g}:\adj_h(x)=\alpha(h)x\text{ for all }h\in\mathfrak{t}\}.$$
One knows that $\mathfrak{g}_0=\mathfrak{t}$, while any non-zero weight $\alpha$ satisfying $\mathfrak{g}_{\alpha}\neq\{0\}$ is called a \textit{root}. In this case, $\mathfrak{g}_{\alpha}$ is one-dimensional and called a \textit{root space}. Furthermore, letting $\Delta\subseteq X^*(T)$ denote the set of roots, we have the following decomposition of $\mathfrak{g}$: \begin{equation}\label{Equation: Root space decomposition}
\mathfrak{g}=\mathfrak{t}\oplus\bigoplus_{\alpha\in\Delta}\mathfrak{g}_{\alpha}.
\end{equation}
One can check that $[\mathfrak{g}_{\alpha},\mathfrak{g}_{\beta}]\subseteq\mathfrak{g}_{\alpha+\beta}$ for all $\alpha,\beta\in\Delta$. In particular, if $\alpha+\beta$ is neither zero nor a root, then $[\mathfrak{g}_{\alpha},\mathfrak{g}_{\beta}]=\{0\}$. 

Our choice of $B$ gives rise to several distinguished subsets of $\Delta$. The positive roots $\Delta_{+}$ are those $\alpha\in\Delta$ for which $\mathfrak{g}_{\alpha}\subseteq\mathfrak{b}$, so that 
$$\mathfrak{b}=\mathfrak{t}\oplus\bigoplus_{\alpha\in\Delta_{+}}\mathfrak{g}_{\alpha}.$$ The negative roots $\Delta_{-}\subseteq\Delta$ are defined by negating the positive roots, ie. $\Delta_{-}:=-\Delta_{+}$. We then have a disjoint union $\Delta=\Delta_{+}\cup\Delta_{-}$. Finally, the simple roots $\Pi\subseteq\Delta$ are positive roots with the property that every $\beta\in\Delta$ can be written as \begin{equation}\label{Equation: Root decomposition}\beta=\sum_{\alpha\in\Pi}c_{\alpha}\alpha\end{equation} for unique coefficients $c_{\alpha}\in\mathbb{Z}$. These coefficients are either all non-negative (in which case $\beta\in\Delta_{+}$) or all non-positive (in which case $\beta\in\Delta_{-}$). Furthermore, the simple roots are known to form a basis of $\mathfrak{t}^*$.   

It will be beneficial to note that the weight lattice carries a particular partial order, defined in terms of the simple roots. Given $\beta,\gamma\in X^*(T)$, one has \begin{equation}\label{Equation: Partial order on the weight lattice}\beta\leq\gamma\Longleftrightarrow\gamma-\beta=\sum_{\alpha\in\Pi}c_{\alpha}\alpha\end{equation} for some strictly non-negative integers $c_{\alpha}\in\mathbb{Z}$.

The Lie algebra $\mathfrak{g}$ carries a distinguished symmetric bilinear form $\langle\cdot,\cdot\rangle:\mathfrak{g}\otimes\mathfrak{g}\rightarrow\mathbb{C}$, called the \textit{Killing form} and defined by 
$$\langle x,y\rangle:=\trace(\adj_x\circ\adj_y),\quad x,y\in\mathfrak{g}.$$ This form is $G$-invariant in the sense that $\langle\Adj_g(x),\Adj_g(y)\rangle=\langle x,y\rangle$ for all $g\in G$ and $x,y\in\mathfrak{g}$. Furthermore, the Killing form is non-degenerate both on $\mathfrak{g}$ and when restricted to a bilinear form on $\mathfrak{t}$. This non-degeneracy gives rise to vector space isomorphisms $\mathfrak{g}\cong\mathfrak{g}^*$ and $\mathfrak{t}\cong\mathfrak{t}^*$, under which the Killing form corresponds to bilinear forms on $\mathfrak{g}^*$ and $\mathfrak{t}^*$, respectively. In an abuse of notation, we will also use $\langle\cdot,\cdot\rangle$ to denote these forms. Each simple root $\alpha\in\Pi$ then determines a \textit{simple coroot} $h_{\alpha}\in\mathfrak{t}$, defined by the following property:
$$\phi(h_{\alpha})=2\frac{\langle\alpha,\phi\rangle}{\langle\alpha,\alpha\rangle}$$
for all $\phi\in\mathfrak{t}^*$. One can check that $h_{\alpha}\in[\mathfrak{g}_{\alpha},\mathfrak{g}_{-\alpha}]$ and that the simple coroots form a basis of $\mathfrak{t}$.

Now let $W:=N_G(T)/T$ be the Weyl group and $s_{\alpha}\in W$ the reflection associated to $\alpha\in\Delta$, so that $W$ is generated by the simple reflections $s_{\alpha},\alpha\in\Pi$. Note that $W$ acts on $\mathfrak{t}$ by $w\cdot x=\Adj_g(x)$, $w\in W$, $x\in\mathfrak{t}$, where $g\in N_G(T)$ is any representative of $w$. There is an induced action on $\mathfrak{t}^*$ defined by $$(w\cdot\phi)(x)=\phi(w^{-1}\cdot x),\quad w\in W, \text{  }\phi\in\mathfrak{t}^*, \text{  }x\in\mathfrak{t},$$ 
under which the reflections act as follows:
$$s_{\alpha}\cdot\phi=\phi-2\frac{\langle\alpha,\phi\rangle}{\langle\alpha,\alpha\rangle}\alpha,\quad\alpha\in\Delta,\text{  }\phi\in\mathfrak{t}^*.$$
It is known that $X^*(T)$ and $\Delta$ are invariant under this $W$-action.

Now recall that a closed subgroup $P\subseteq G$ is called \textit{parabolic} if $P$ contains a conjugate of $B$. If $P$ contains $B$ itself, one calls it a \textit{standard parabolic} subgroup. To construct such a subgroup, consider a subset $S\subseteq\Pi$ and let $\Delta_{S}$ (resp. $\Delta_{S}^+$, $\Delta_{S}^{-}$) denote the collection of roots (resp. positive roots, negative roots) expressible as $\mathbb{Z}$-linear combinations of the elements of $S$. It follows that \begin{equation}\label{Equation: Lie algebra of standard parabolic}\mathfrak{p}_S:=\mathfrak{b}\oplus\bigoplus_{\alpha\in\Delta_{S}^{-}}\mathfrak{g}_{\alpha}\end{equation}
is a Lie subalgebra of $\mathfrak{g}$, and we shall let $P_S\subseteq G$ denote the corresponding closed connected subgroup of $G$. Note that $P_S$ contains $B$, by construction. Moreover, it turns out that $S\mapsto P_S$ defines a bijective correspondence between the subsets of $\Pi$ and the standard parabolic subgroups of $G$. The inverse associates to a standard parabolic subgroup $P$ the subset $\Pi_P:=\{\alpha\in\Pi:\mathfrak{g}_{-\alpha}\subseteq\mathfrak{p}\}$, where $\mathfrak{p}$ is the Lie algebra of $P$.

We will later benefit from recalling a particular Levi decomposition of a standard parabolic subgroup $P_S$ (cf. \cite[Sect. 30.2]{Humphreys}). One has \begin{equation}\label{Equation: Levi decomposition of parabolic} P_S=U_S\rtimes L_S,\end{equation} where $U_S$ is the unipotent radical of $P_S$ and $L_S$ is a (reductive) Levi factor having Lie algebra 
\begin{equation}\label{Equation: Lie algebra of Levi factor}
\mathfrak{l}_S:=\mathfrak{t}\oplus\bigoplus_{\alpha\in\Delta_S}\mathfrak{g}_{\alpha}=\left(\bigoplus_{\alpha\in\Delta_S^{-}}\mathfrak{g}_{\alpha}\right)\oplus\mathfrak{t}\oplus\left(\bigoplus_{\alpha\in\Delta_S^+}\mathfrak{g}_{\alpha}\right).\end{equation}
The Lie algebra of $U_{S}$ is 
\begin{equation}\label{Equation: Lie algebra of unipotent radical}
\mathfrak{u}_{S}:=\bigoplus_{\alpha\in\Delta_{+}\setminus\Delta_{S}^+}\mathfrak{g}_{\alpha}.
\end{equation}

We now specialize some of the above Lie-theoretic generalities to the case of $G=\SL_n(\mathbb{C})$, a recurring example in this article.

\begin{example}\label{Example: The type A setup}
Let $G=\SL_n(\mathbb{C})$, whose Lie algebra is $\mathfrak{g}=\mathfrak{sl}_n(\mathbb{C})=\{x\in\mathfrak{gl}_n(\mathbb{C}):\trace(x)=0\}$ with Lie bracket the commutator of matrices. The adjoint representations are given by
\begin{equation}\label{Equation: Type A adjoint representation}
\Adj_g(x)=gxg^{-1},\quad\adj_x(y)=[x,y]=xy-yx
\end{equation}
for all $g\in\SL_n(\mathbb{C})$ and $x,y\in\mathfrak{sl}_n(\mathbb{C})$. Also, the Killing form $\langle\cdot,\cdot\rangle$ is given by
\begin{equation}\label{Equation: Type A Killing form}
\langle x,y\rangle=2n\trace(xy),\quad x,y\in\mathfrak{sl}_n(\mathbb{C}).
\end{equation}
One may take $T$ and $B$ to be the subgroups of diagonal and upper-triangular matrices in $\SL_n(\mathbb{C})$, respectively. It follows that $\mathfrak{t}$ and $\mathfrak{b}$ are the Lie subalgebras of diagonal and upper-triangular matrices in $\mathfrak{sl}_n(\mathbb{C})$, respectively. Now, for each $i\in\{1,\ldots,n\}$, let $t_i\in X^*(T)$ be the weight defined by
$$t_i:T\rightarrow\mathbb{C}^*,\quad\begin{bmatrix} a_1 & 0 & \ldots & 0\\ 0 & a_2 & \ldots & 0\\ \vdots & \vdots & \ddots & \vdots\ \\ 0 & 0 & \ldots & a_n\end{bmatrix}\mapsto a_i.$$ We then have $\Delta=\{t_i-t_j:1\leq i,j\leq n,i\neq j\}$, $\Delta_{+}=\{t_i-t_j:1\leq i<j\leq n\}$, $\Delta_{-}=\{t_i-t_j:1\leq j<i\leq n\},$ and $\Pi=\{t_i-t_{i+1}:1\leq i\leq n-1\}$. If $\alpha=t_i-t_j\in\Delta$, then $\mathfrak{sl}_n(\mathbb{C})_{\alpha}=\text{span}_{\mathbb{C}}\{E_{ij}\}$ where $E_{ij}\in\mathfrak{sl}_n(\mathbb{C})$ has an entry of $1$ in position $(i,j)$ and 0 elsewhere.

Recall that the Weyl group identifies with $S_n$ in such a way that $w\in S_n$ acts on $h\in\mathfrak{t}$ by permuting entries along the diagonal. 
\end{example}
  
\subsection{Semisimple and nilpotent elements}\label{Subsection: semisimple and nilpotent elements}
It will be prudent to recall some of the fundamentals concerning semisimple and nilpotent elements in $\mathfrak{g}$. We begin with the official, Lie-theoretic definitions of these terms.

\begin{definition}\label{Definition: semisimple and nilpotent}
We call a point $x\in\mathfrak{g}$ \textit{semisimple} (resp. \textit{nilpotent}) if $\adj_x:\mathfrak{g}\rightarrow\mathfrak{g}$ is semisimple (resp. nilpotent) as a vector space endomorphism.
\end{definition}

At first glance, there might appear to be other legitimate definitions. Suppose, for instance, that $\mathfrak{g}$ is explicitly presented as a subalgebra of $\mathfrak{gl}_n(\mathbb{C})$ for some $n$. Elements of $\mathfrak{g}$ are then $n\times n$ matrices, and one has the usual notions of ''semisimple" and ''nilpotent" for square matrices. Fortunately, these turn out to coincide with Definition \ref{Definition: semisimple and nilpotent} (see \cite[Cor. 20.4.3]{Tauvel}).

\begin{theorem}\label{Theorem: equivalence of semisimple / nilpotent}
If $\mathfrak{g}$ is a subalgebra of $\mathfrak{gl}_n(\mathbb{C})$, then $x\in\mathfrak{g}$ is semisimple (resp. nilpotent) in the sense of Definition \ref{Definition: semisimple and nilpotent} if and only if $x$ is semisimple (resp. nilpotent) as a matrix.
\end{theorem}

\begin{example}\label{Example: sl2 triple}
Suppose that $G=\SL_2(\mathbb{C})$, whose Lie algebra $\mathfrak{g}=\mathfrak{sl}_2(\mathbb{C})$ has the usual generators \begin{equation}\label{Equation: Usual sl2 generators}X=\begin{bmatrix}0 & 1\\ 0 & 0\end{bmatrix}, \quad H=\begin{bmatrix}1 & 0\\ 0 & -1\end{bmatrix},\quad Y=\begin{bmatrix}0 & 0\\ 1 & 0\end{bmatrix}.\end{equation} By Theorem \ref{Theorem: equivalence of semisimple / nilpotent}, it is immediate that $H$ is a semisimple element of $\mathfrak{sl}_2(\mathbb{C})$, while $X$ and $Y$ are nilpotent elements. The algebra $\mathfrak{sl}_2(\mathbb{C})$ and triple of elements $(X,H,Y)$ will play a substantial role in \ref{Subsection: The Jacobson-Morozov Theorem and sl2-triples}. 
\end{example}

The relationship between semisimple (resp. nilpotent) elements of $\mathfrak{g}$ and semisimple (resp. nilpotent) matrices extends beyond Theorem \ref{Theorem: equivalence of semisimple / nilpotent}. For instance, recall that if an $n\times n$ matrix is both semisimple and nilpotent, it is necessarily zero. At the same time, if $x\in\mathfrak{g}$ is both semisimple and nilpotent, so too is the endomorphism $\text{ad}_x:\mathfrak{g}\rightarrow\mathfrak{g}$. It follows that $\adj_x=0$, and the injectivity of $\text{ad}:\mathfrak{g}\rightarrow\mathfrak{gl}(\mathfrak{g})$ implies $x=0$. For a second comparison, note that an $n\times n$ matrix is semisimple (resp. nilpotent) if and only if it is conjugate to a diagonal (resp. strictly upper-triangular) matrix. Dually, $x\in\mathfrak{g}$ is semisimple (resp. nilpotent) if and only if there exists $g\in G$ satisfying $\Adj_g(x)\in\mathfrak{t}$ (resp. $\Adj_g(x)\in\bigoplus_{\alpha\in\Delta_{+}}\mathfrak{g}_{\alpha}$). Thirdly, one knows that every $n\times n$ matrix $A$ gives rise to unique $n\times n$ matrices $S$ and $N$ such that $S$ is semisimple, $N$ is nilpotent, $A=S+N$, and $SN-NS=0$. The analogous fact is that for every $x\in\mathfrak{g}$, there exist unique elements $x_s,x_n\in\mathfrak{g}$ for which $x_s$ is semisimple, $x_n$ is nilpotent, $x=x_s+x_n$, and $[x_s,x_n]=0$. The expression $x=x_s+x_n$ is called the \textit{Jordan decomposition} of $x$, while $x_s$ and $x_n$ are called the semisimple and nilpotent parts of $x$, respectively.     

Let us turn to a more geometric discussion. Denote by $\mathfrak{g}_{\text{ss}}$ and $\mathcal{N}$ the subsets of semisimple and nilpotent elements of $\mathfrak{g}$, respectively. The latter is called the \textit{nilpotent cone}, in reference to the fact that $\mathcal{N}$ is invariant under the dilation action of $\mathbb{C}^*$ on $\mathfrak{g}$. In the interest of preparing for later sections of this article, let us mention a few preliminary facts concerning $\mathfrak{g}_{ss}$ and $\mathcal{N}$. Firstly, since the semisimple endomorphisms are open and dense in $\mathfrak{gl}(\mathfrak{g})$, it follows that $\mathfrak{g}_{\text{ss}}$ is an open dense subvariety of $\mathfrak{g}$. One analogously sees that $\mathcal{N}$ is a closed subvariety of $\mathfrak{g}$. Secondly, noting that $\Adj_g$ is a Lie algebra automorphism for each $g\in G$, we conclude that $x\in\mathfrak{g}$ is semisimple (resp. nilpotent) if and only if $\Adj_g(x)$ is semisimple (resp. nilpotent). This is equivalent to the observation that both $\mathfrak{g}_{\text{ss}}$ and $\mathcal{N}$ are invariant under the adjoint representation.

\begin{example}
Let $G=\SL_2(\mathbb{C})$. By Theorem \ref{Theorem: equivalence of semisimple / nilpotent}, $A\in\mathfrak{sl}_2(\mathbb{C})$ is nilpotent if and only if it is nilpotent as a matrix. Equivalently, the characteristic polynomial $\det(\lambda I_2-A)$ is $\lambda^2$, which is the case if and only if $\det(A)=0$ (the condition $\trace(A)=0$ being automatic). In other words, we have 
$$\mathcal{N}=\left\{\begin{bmatrix} x & y \\ z & -x\end{bmatrix}\in\mathfrak{sl}_2(\mathbb{C}):x^2+yz=0\right\}.$$
\end{example}  
     
\subsection{The Jacobson-Morozov Theorem and $\mathfrak{sl}_2(\mathbb{C})$-triples}\label{Subsection: The Jacobson-Morozov Theorem and sl2-triples}
Let $(X,H,Y)$ be the triple of matrices in $\mathfrak{sl}_2(\mathbb{C})$ from Example \ref{Example: sl2 triple}. These matrices form a $\mathbb{C}$-basis of $\mathfrak{sl}_2(\mathbb{C})$ and satisfy the following bracket relations:
\begin{equation}\label{Equation: sl2 relations}
[X,Y]=H,\quad [H,X]=2X, \quad [H,Y]=-2Y.
\end{equation}
More generally, one calls a triple $(x,h,y)$ of points in $\mathfrak{g}$ an $\mathfrak{sl}_2(\mathbb{C})$-\textit{triple} if \eqref{Equation: sl2 relations} holds in $\mathfrak{g}$ when $X$, $H$, and $Y$ are replaced with $x$, $h$, and $y$, respectively. Such triples are in bijective correspondence with the Lie algebra morphisms $\phi:\mathfrak{sl}_2(\mathbb{C})\rightarrow\mathfrak{g}$, where $\phi\in\Hom(\mathfrak{sl}_2(\mathbb{C}),\mathfrak{g})$ identifies with $(\phi(X),\phi(H),\phi(Y))$. Since any $\phi\in\Hom(\mathfrak{sl}_2(\mathbb{C}),\mathfrak{g})$ is either zero or injective (a consequence of $\mathfrak{sl}_2(\mathbb{C})$ having no non-trivial ideals), there are exactly two possibilities for an $\mathfrak{sl}_2(\mathbb{C})$-triple $(x,h,y)$. Either $x=h=y=0$ or $\text{span}_{\mathbb{C}}\{x,h,y\}$ is a subalgebra of $\mathfrak{g}$ isomorphic to $\mathfrak{sl}_2(\mathbb{C})$. 

\begin{example}\label{Example: Triple}
Let $\alpha\in\Pi$ be a simple root and $h_{\alpha}\in[\mathfrak{g}_{\alpha},\mathfrak{g}_{-\alpha}]$ the corresponding simple coroot (as defined in \ref{Subsection: The basic objects}). If $e_{\alpha}\in\mathfrak{g}_{\alpha}$ and $e_{-\alpha}\in\mathfrak{g}_{-\alpha}$ are chosen so that $h_{\alpha}=[e_{\alpha},e_{-\alpha}]$, then one can show $(e_{\alpha},h_{\alpha},e_{-\alpha})$ to be an $\mathfrak{sl}_2(\mathbb{C})$-triple.
\end{example}

\begin{example}\label{Example: Kostant triple}
In the interest of preparing for discussions to come (particularly \ref{Subsection: The regular nilpotent orbit}), we now consider a more complicated example (cf. \cite[Lemma 5.2]{Kostant}). To this end, since $\Pi$ is a basis of $\mathfrak{t}^*$, we may define $\{\epsilon_{\alpha}\}_{\alpha\in\Pi}$ to be the basis of $\mathfrak{t}$ satisfying 
\begin{equation}\label{Equation: Dual basis}
\alpha(\epsilon_{\beta}) = 
\begin{cases} 
2 & \text{if } \alpha=\beta \\
0 & \text{if } \alpha\neq\beta
\end{cases}
\end{equation}
for all $\alpha,\beta\in\Pi$. Now consider the semisimple element 
\begin{equation}\label{Equation: Regular semisimple element}h:=\sum_{\alpha\in\Pi}\epsilon_{\alpha}\in\mathfrak{t}.\end{equation} Noting that $\{h_{\alpha}\}_{\alpha\in\Pi}$ is a basis of $\mathfrak{t}$ (see \ref{Subsection: The basic objects}), there exist unique coefficients $c_{\alpha}\in\mathbb{C}$, $\alpha\in\Pi$, such that
$$h=\sum_{\alpha\in\Pi}c_{\alpha}h_{\alpha}.$$ We now define $\xi,\eta\in\mathfrak{g}$ by 
$$\xi:=\sum_{\alpha\in\Pi}e_{\alpha}$$ and
$$\eta:=\sum_{\alpha\in\Pi}c_{\alpha}e_{-\alpha},$$ respectively, where for each $\alpha\in\Pi$, $e_{\alpha}\in\mathfrak{g}_{\alpha}$ and $e_{-\alpha}\in\mathfrak{g}_{-\alpha}$ are chosen such that $h_{\alpha}=[e_{\alpha},e_{-\alpha}]$ (as in Example \ref{Example: Triple}). Furthermore, we claim that $(\xi,h,\eta)$ is an $\mathfrak{sl}_2(\mathbb{C})$-triple. To see this, first note that \begin{equation}\label{Equation: Technical calculation}[\xi,\eta]=\sum_{\alpha,\beta\in\Pi}c_{\beta}[e_{\alpha},e_{-\beta}].\end{equation} If $\alpha,\beta\in\Pi$ are distinct, then the discussion of simple roots in \ref{Subsection: The basic objects} implies that $\alpha-\beta$ cannot be a root. Since $\alpha-\beta$ is also non-zero in this case, $[\mathfrak{g}_{\alpha},\mathfrak{g}_{-\beta}]=\{0\}$ (see \ref{Subsection: The basic objects}). In particular, $[e_{\alpha},e_{-\beta}]=0$ when $\alpha$ and $\beta$ are distinct. At the same time, we know that $[e_{\alpha},e_{-\alpha}]=h_{\alpha}$ for all $\alpha\in\Pi$. Hence \eqref{Equation: Technical calculation} can be written as $$[\xi,\eta]=\sum_{\alpha\in\Pi}c_{\alpha}h_{\alpha}=h.$$ Now computing $[h,\xi]$, we have
$$[h,\xi]=\sum_{\alpha\in\Pi}[h,e_{\alpha}]=\sum_{\alpha\in\Pi}\alpha(h)e_{\alpha}.$$ The description \eqref{Equation: Regular semisimple element} of $h$ implies that $\alpha(h)=2$ for all $\alpha\in\Pi$, so that
$$[h,\xi]=\sum_{\alpha\in\Pi}2e_{\alpha}=2\xi.$$ The verification of $[h,\eta]=-2\eta$ is similar, since $-\alpha(h)=-2$ for all $\alpha\in\Pi$. It follows that $(\xi,h,\eta)$ is indeed an $\mathfrak{sl}_2(\mathbb{C})$-triple.  
\end{example}

We will benefit from understanding $\mathfrak{sl}_2(\mathbb{C})$-triples in the context of the \textit{Jacobson-Morozov Theorem}, a crucial result that will feature prominently in \ref{Subsection: Some first results}. To motivate this theorem, suppose that $(x,h,y)$ is the $\mathfrak{sl}_2(\mathbb{C})$-triple corresponding to a non-zero (ie. injective) $\phi\in\Hom(\mathfrak{sl}_2(\mathbb{C}),\mathfrak{g})$. It follows that $\phi$ defines an isomorphism with its image $\text{span}_{\mathbb{C}}\{x,h,y\}$, so that $\phi$ sends nilpotent elements of $\mathfrak{sl}_2(\mathbb{C})$ to nilpotent elements of $\text{span}_{\mathbb{C}}\{x,h,y\}$. In particular, $x=\phi(X)$ is nilpotent in $\text{span}_{\mathbb{C}}\{x,h,y\}$. If one chooses a Lie algebra embedding $\mathfrak{g}\subseteq\mathfrak{gl}_n(\mathbb{C})$\footnote{Since $\mathfrak{g}$ is semisimple, such an embedding always exists. For instance, one may identify $\mathfrak{g}$ with its image under the (injective) adjoint representation $\adj:\mathfrak{g}\rightarrow\mathfrak{gl}(\mathfrak{g})\cong\mathfrak{gl}_n(\mathbb{C})$, $n=\dim(\mathfrak{g})$.}, then $\text{span}_{\mathbb{C}}\{x,h,y\}$ becomes a subalgebra of $\mathfrak{gl}_n(\mathbb{C})$ and Theorem \ref{Theorem: equivalence of semisimple / nilpotent} implies that $x$ is nilpotent matrix. By applying Theorem \ref{Theorem: equivalence of semisimple / nilpotent} again, this time to the subalgebra $\mathfrak{g}\subseteq\mathfrak{gl}_n(\mathbb{C})$, we see that $x$ is nilpotent in $\mathfrak{g}$. This is clearly also true when $\phi=0$ (in which case $x=0$), so we have shown $x\in\mathfrak{g}$ to be nilpotent whenever it appears as the first element in an $\mathfrak{sl}_2(\mathbb{C})$-triple. For a converse, one has the Jacobson-Morozov Theorem.

\begin{theorem}[The Jacobson-Morozov Theorem; \text{cf. \cite[Thm. 3.7.1]{Chriss}}]
If $x\in\mathfrak{g}$ is nilpotent, then there exist $h,y\in\mathfrak{g}$ such that $(x,h,y)$ is an $\mathfrak{sl}_2(\mathbb{C})$-triple.
\end{theorem}  

\section{General adjoint orbits}\label{Section: General adjoint orbits}
Equipped with Section \ref{Section: Preliminaries on Lie theory}, we can apply parts of Section \ref{Section: Preliminaries on algebraic group actions} to study a distinguished class of algebraic group orbits -- adjoint orbits. Our overarching objective is to give the background and context necessary for the more specialized discussions of semisimple and nilpotent adjoint orbits (Sections \ref{Section: Semisimple orbits} and \ref{Section: Nilpotent orbits}, respectively). Nevertheless, we will see that general adjoint orbits are interesting objects of study in Lie theory and geometry (see \ref{Subsection: Geometric features}, for example).     

\subsection{Definitions and conventions}\label{Subsection: Definitions and conventions} 
Following the framework of Example \ref{Example: action from representation}, the adjoint representation $\Adj:G\rightarrow\GL(\mathfrak{g})$ induces a $G$-variety structure on $\mathfrak{g}$. One calls the resulting action the \textit{adjoint action} and its orbits \textit{adjoint orbits}. 

\begin{example}\label{Example: Type A adjoint orbits}
Let $G=\SL_n(\mathbb{C})$. Referring to \eqref{Equation: Type A adjoint representation}, we see that the adjoint orbits of $\SL_n(\mathbb{C})$ are precisely the conjugacy classes of the traceless $n\times n$ matrices. Note that the term ``conjugacy class'' is unambiguous, since two $n\times n$ matrices are conjugate under $\SL_n(\mathbb{C})$ if and only if they are conjugate under $\GL_n(\mathbb{C})$. With this point in mind, Example \ref{Example: Conjugation} describes several adjoint $\SL_n(\mathbb{C})$-orbits.\footnote{To ensure that $\mathcal{O}(A)$ from Example \ref{Example: Conjugation} is an adjoint $\SL_n(\mathbb{C})$-orbit, one must impose the extra condition $\trace(A)=0$.}
\end{example}

Now, recall the discussion of stabilizers from \ref{Subsection: Orbits as quotient varieties}. Given $x\in\mathfrak{g}$, we shall always understand $C_G(x)$ as referring to the $G$-stabilizer of $x$ under the adjoint action, ie. $$C_G(x):=\{g\in G:\Adj_g(x)=x\},\quad x\in\mathfrak{g}.$$ The following fact about its Lie algebra, $C_{\mathfrak{g}}(x)\subseteq\mathfrak{g}$, will be used repeatedly and (in the interest of parsimony) without explicit indication that it is being used.  

\begin{lemma}[\text{cf. \cite[Sect. 2.2]{Jantzen}}]
If $x\in\mathfrak{g}$, then $C_{\mathfrak{g}}(x)=\{y\in\mathfrak{g}:[x,y]=0\}$.
\end{lemma}

\subsection{Dimension and regularity}\label{Subsection: Dimension and regularity}
In the interest of later sections, we now take a moment to discuss the dimensions of adjoint orbits. One might begin with the observation that each adjoint orbit can have dimension at most $\dim(G)$. However, this dimension bound turns out to be highly sub-optimal. To improve it, we note that $\dim(C_G(x))\geq\text{rank}(G)$ for all $x\in\mathfrak{g}$ (see \cite[Sect. 19.7, Sect. 29.3]{Tauvel}\footnote{Strictly speaking, this reference bounds the dimension of $C_{\mathfrak{g}}(x)$ instead of $C_G(x)$. This is, of course, equivalent to the dimension inequality stated here.}). It follows that
$$\dim(\mathcal{O}(x))=\dim(G/C_G(x))=\dim(G)-\dim(C_G(x))\leq\dim(G)-\text{rank}(G)$$         
for all $x\in\mathfrak{g}$, so that $\dim(G)-\text{rank}(G)$ is our new upper bound on adjoint orbit dimensions. This bound is sharp in the sense that $\mathfrak{g}$ always contains adjoint orbits of dimension $\dim(G)-\text{rank}(G)$. We defer the proof of this fact to \ref{Subsection: Stabilizer descriptions} (see Corollary \ref{Corollary: The existence of regular orbits}), which is completely independent of the present section.

One calls an adjoint orbit $\mathcal{O}\subseteq\mathfrak{g}$ \textit{regular} if $\dim(\mathcal{O})=\dim(G)-\text{rank}(G)$.     

\subsection{Some first results}\label{Subsection: Some first results}
We now gather some initial results on adjoint orbits, emphasizing connections to semisimple/nilpotent elements and the Jordan decomposition (discussed in \ref{Subsection: semisimple and nilpotent elements}). The facts we establish here will be essential to Sections \ref{Section: Semisimple orbits} and \ref{Section: Nilpotent orbits}.

\begin{proposition}\label{Proposition: Dilating the nilpotent part}
	If $\mathcal{O}\subseteq\mathfrak{g}$ is an adjoint orbit and $x\in\mathcal{O}$ has Jordan decomposition $x=x_s+x_n$, then $x_s+ax_n\in\mathcal{O}$ for all $a\in\mathbb{C}^*$.
\end{proposition}

\begin{proof}
	Since $[x_s,x_n]=0$, it follows that $x_n\in C_{\mathfrak{g}}(x_s)$. However, $C_{\mathfrak{g}}(x_s)$ is known to be a reductive subalgebra of $\mathfrak{g}$ (see \cite[Prop. 20.5.13]{Tauvel}), and one can then deduce that $x_n\in [C_{\mathfrak{g}}(x_s),C_{\mathfrak{g}}(x_s)]$ (see \cite[Prop. 20.5.14]{Tauvel}). Also, the fact that $[C_{\mathfrak{g}}(x_s),C_{\mathfrak{g}}(x_s)]$ is reductive implies that $[C_{\mathfrak{g}}(x_s),C_{\mathfrak{g}}(x_s)]$ is a semisimple subalgebra. The element $x_n$ is nilpotent in $[C_{\mathfrak{g}}(x_s),C_{\mathfrak{g}}(x_s)]$, as one can see by embedding $\mathfrak{g}$ (hence also $[C_{\mathfrak{g}}(x_s),C_{\mathfrak{g}}(x_s)]$) into $\mathfrak{gl}_n(\mathbb{C})$ and applying Theorem \ref{Theorem: equivalence of semisimple / nilpotent}. By the Jacobson-Morozov Theorem, there exists $h\in [C_{\mathfrak{g}}(x_s),C_{\mathfrak{g}}(x_s)]$ for which $[h,x_n]=2x_n$. Now for all $c\in\mathbb{C}$, we have
	\begin{equation}\label{Equation: First dilation invariance}\Adj_{\exp(ch)}(x)=\exp(\adj_{ch})(x)=\exp(\adj_{ch})(x_s)+\exp(\adj_{ch})(x_n)
	\end{equation}
	where $\exp(\adj_{ch})$ denotes the exponential of the endomorphism $\adj_{ch}:\mathfrak{g}\rightarrow\mathfrak{g}$. Observe that $ch\in C_{\mathfrak{g}}(x_s)$, so that $\exp(\adj_{ch})(x_s)=x_s$. Secondly, the condition $[ch,x_n]=2cx_n$ gives $\exp(\adj_{ch})(x_n)=e^{2c}x_n$. These two observations allow one to write \eqref{Equation: First dilation invariance} as
	$$\Adj_{\exp(ch)}(x)=x_s+e^{2c}x_n.$$
	In particular, $x_s+e^{2c}x_n\in\mathcal{O}$ for all $c\in\mathbb{C}$. To prove the proposition as stated, choose $c\in\mathbb{C}$ such that $e^{2c}=a$.  
\end{proof}

By letting $a\rightarrow 0$ in Proposition \ref{Proposition: Dilating the nilpotent part}, one obtains the following corollary.

\begin{corollary}\label{Corollary: Semisimple part in closure}
	If $\mathcal{O}\subseteq\mathfrak{g}$ is an adjoint orbit and $x\in\mathcal{O}$ has Jordan decomposition $x=x_s+x_n$, then $x_s\in\overline{\mathcal{O}}$.
\end{corollary}

Our next results will make extensive use of $G$-invariant polynomials on $\mathfrak{g}$. To be more precise, note that the adjoint action induces an action of $G$ on $\mathbb{C}[\mathfrak{g}]$ (cf. Example \ref{Example: Induced action on functions}), the algebra of regular (ie. polynomial) functions on $\mathfrak{g}$. Let $\mathbb{C}[\mathfrak{g}]^G\subseteq\mathbb{C}[\mathfrak{g}]$ denote the subalgebra of all $G$-invariant polynomial functions, meaning
$$\mathbb{C}[\mathfrak{g}]^G:=\{f\in\mathbb{C}[\mathfrak{g}]:g\cdot f=f\text{ for all }g\in G\}.$$ Analogously, the $W$-action on $\mathfrak{t}$ (see \ref{Subsection: The basic objects}) gives rise to a $W$-action on the polynomial functions $\mathbb{C}[\mathfrak{t}]$, and to a subalgebra $\mathbb{C}[\mathfrak{t}]^W\subseteq\mathbb{C}[\mathfrak{t}]$ of $W$-invariant polynomial functions. One can verify that $G$-invariant polynomials on $\mathfrak{g}$ restrict to $W$-invariant polynomials on $\mathfrak{t}$, and it turns out that this restriction process defines an algebra isomorphism \begin{equation}\label{Equation: Chevalley restriction}
\mathbb{C}[\mathfrak{g}]^G\xrightarrow{\cong}\mathbb{C}[\mathfrak{t}]^W
\end{equation}
(see \cite[Thm. 3.1.38]{Chriss}).
Moreover, $\mathbb{C}[\mathfrak{g}]^G$ (hence also $\mathbb{C}[\mathfrak{t}]^W$) is known to be freely generated as a commutative $\mathbb{C}$-algebra by $r=\text{rank}(G)$ homogeneous polynomials (see \cite[Sect. 3.3]{KostantPoly}).    

\begin{proposition}\label{Proposition: Invariant functions have constant values}
	If $x,y\in\mathfrak{g}$ are semisimple, then $\mathcal{O}(x)=\mathcal{O}(y)$ if and only if $f(x)=f(y)$ for all $f\in\mathbb{C}[\mathfrak{g}]^G$.
\end{proposition}

\begin{proof}
	If $\mathcal{O}(x)=\mathcal{O}(y)$, then $y=\Adj_g(x)$ for some $g\in G$ and $f(y)=f(\Adj_g(x))=f(x)$ for all $f\in\mathbb{C}[\mathfrak{g}]^G$. Conversely, suppose that $f(x)=f(y)$ for all $f\in\mathbb{C}[\mathfrak{g}]^G$. Since $x$ and $y$ are semisimple, there exist $h_x,h_y\in\mathfrak{t}$ such that $\mathcal{O}(h_x)=\mathcal{O}(x)$ and $\mathcal{O}(h_y)=\mathcal{O}(y)$ (see the discussion in \ref{Subsection: semisimple and nilpotent elements}). Also, repeating the argument from the first sentence of our proof, we see $f(h_x)=f(x)$ and $f(h_y)=f(y)$ for all $f\in\mathbb{C}[\mathfrak{g}]^G$. We may therefore assume that $x,y\in\mathfrak{t}$ when proving our converse. It then follows from \eqref{Equation: Chevalley restriction} that $f(x)=f(y)$ for all $f\in\mathbb{C}[\mathfrak{t}]^W$. Using Lemma 34.2.1 of \cite{Tauvel}, we conclude that $y=w\cdot x$ for some $w\in W$. After lifting $w$ to a representative $g\in N_G(T)$, this statement becomes $y=\Adj_g(x)$. Hence $\mathcal{O}(x)=\mathcal{O}(y)$.     
\end{proof}

Focusing now on nilpotent elements and adjoint orbits, we have the following.

\begin{proposition}\label{Proposition: Characterizations of nilpotent elements}
	If $x\in\mathfrak{g}$, then the following conditions are equivalent.
	\begin{itemize}
		\item[(i)] $x$ is nilpotent.
		\item[(ii)] For all $a\in\mathbb{C}^*$, $ax\in\mathcal{O}(x)$.
		\item[(iii)] $0\in\overline{\mathcal{O}(x)}$.
		\item[(iv)] For all $f\in\mathbb{C}[\mathfrak{g}]^G$, $f(x)=f(0)$.
	\end{itemize}
\end{proposition}

\begin{proof}
We will verify the implications (i)$\Rightarrow$(ii)$\Rightarrow$(iii)$\Rightarrow$(iv)$\Rightarrow$(i).

\noindent\underline{(i)$\Rightarrow$(ii):} Note that $x_s=0$ and $x_n=x$, so that (ii) follows immediately from Proposition \ref{Proposition: Dilating the nilpotent part}. 

\noindent\underline{(ii)$\Rightarrow$(iii):} Since $ax\in\mathcal{O}(x)$ for all $a\in\mathbb{C}^*$, letting $a\rightarrow 0$ establishes that $0\in\overline{\mathcal{O}(x)}$.

\noindent\underline{(iii)$\Rightarrow$(iv):} If $f\in\mathbb{C}[\mathfrak{g}]^G$, then $f(\Adj_g(x))=f(x)$ for all $g\in G$. In other words, $f$ takes the constant value $f(x)$ on $\mathcal{O}(x)$. By continuity, $f$ actually takes the constant value $f(x)$ on $\overline{\mathcal{O}(x)}$. Since $0\in\overline{\mathcal{O}(x)}$, we see that $f(x)=f(0)$.

\noindent\underline{(iv)$\Rightarrow$(i):} For each positive integer $n$, consider the function $\phi_n\in\mathbb{C}[\mathfrak{g}]$ defined by $\phi_n(y)=\trace((\adj_y)^n)$. To see that $\phi_n\in\mathbb{C}[\mathfrak{g}]^G$, suppose that $g\in G$ and $y,z\in\mathfrak{g}$. We have $\adj_{\Adj_g(y)}(z)=[\Adj_g(y),z]=\Adj_g([y,(\Adj_{g})^{-1}(z)])$, meaning that $\adj_{\Adj_g(y)}=\Adj_g\circ\adj_y\circ(\Adj_g)^{-1}$. Hence, $$\phi_n(\Adj_g(y))=\trace((\adj_{\Adj_g(y)})^n)=\trace(\Adj_g\circ(\adj_y)^n\circ(\Adj_g)^{-1})=\trace((\adj_y)^n)=\phi_n(y),$$ and we conclude that $\phi_n\in \mathbb{C}[\mathfrak{g}]^G$. Now by hypothesis, $$\trace((\adj_x)^n)=\phi_n(x)=\phi_n(0)=0$$ for all positive integers $n$. It follows that $\adj_x$ is a nilpotent endomorphism, as desired. 
\end{proof}

\subsection{The adjoint quotient}\label{Subsection: The adjoint quotient}
Adjoint orbits feature prominently in a very natural and well-studied fibration whose total space is $\mathfrak{g}$. To construct it, one can appeal to the basics of geometric invariant theory. The inclusion $\mathbb{C}[\mathfrak{g}]^G\hookrightarrow\mathbb{C}[\mathfrak{g}]$ corresponds to a $G$-invariant surjective map of affine varieties,
\begin{equation}\label{Equation: The adjoint quotient}
\Phi:\mathfrak{g}\rightarrow\Spec(\mathbb{C}[\mathfrak{g}]^G),
\end{equation}
(see \cite[Sect. 1.4.1, Sect. 1.4.2]{Schmitt}), 
called the \textit{adjoint quotient}. Note $\Spec(\mathbb{C}[\mathfrak{g}]^G)$ refers to the maximal ideal spectrum of $\mathbb{C}[\mathfrak{g}]^G$, while $G$-invariance is the condition that $\phi\circ\Adj_g=\phi$ for all $g\in G$.   

One can say a great deal more about \eqref{Equation: The adjoint quotient} by invoking the discussion of $\mathbb{C}[\mathfrak{g}]^G$ in \ref{Subsection: Some first results}. Indeed, recall that $\mathbb{C}[\mathfrak{g}]^G$ is freely generated by $r:=\text{rank}(G)$ homogeneous generators, $\chi_1,\chi_2,\ldots,\chi_r\in\mathbb{C}[\mathfrak{g}]^G$. We may therefore identify $\Spec(\mathbb{C}[\mathfrak{g}]^G)$ with $\mathbb{C}^r$ and re-write $\Phi$ as the map 
\begin{equation}\label{Equation: The second adjoint quotient}
\Phi:\mathfrak{g}\rightarrow\mathbb{C}^r,\quad \xi\mapsto(\chi_1(\xi),\chi_2(\xi),\ldots,\chi_r(\xi)).
\end{equation}
Kostant studied this form of the adjoint quotient in detail (see \cite{KostantPoly}). More specifically, note that the $G$-invariance of $\Phi$ is equivalent to each fibre of $\Phi$ being a union of adjoint orbits. With this in mind, Kostant gave the following description of each fibre's decomposition into adjoint orbits (see \cite[Thm. 0.6, Thm. 0.7]{KostantPoly}). 

\begin{theorem}\label{Theorem: Kostant's theorem}
If $z\in\Spec(\mathbb{C}[\mathfrak{g}]^G)$, then $\Phi^{-1}(z)$ is a union of finitely many adjoint orbits and contains a unique regular orbit. This regular orbit is open and dense in $\Phi^{-1}(z)$.
\end{theorem}

\begin{corollary}
If $\mathcal{O}\subseteq\mathfrak{g}$ is an adjoint orbit, then $\overline{\mathcal{O}}$ is a union of finitely many adjoint orbits.
\end{corollary}

\begin{proof}
Choose a point $x\in\mathcal{O}$. Since the fibres of $\Phi$ are $G$-invariant, it follows that $\mathcal{O}$ belongs to $\Phi^{-1}(\Phi(x))$. This fibre is closed, meaning that $\overline{\mathcal{O}}\subseteq\Phi^{-1}(\Phi(x))$. Theorem \ref{Theorem: Kostant's theorem} implies that $\Phi^{-1}(\Phi(x))$ is a union of finitely many adjoint orbits, so that the same must be true of $\overline{\mathcal{O}}$.  
\end{proof}    

\subsection{Geometric features}\label{Subsection: Geometric features}
We now turn our attention to some more inherently geometric aspects of adjoint orbits. However, to avoid the lengthy digressions needed to properly motivate certain definitions, we will sometimes not define the geometric structures under consideration (ex. Poisson and hyperk\"{a}hler structures). In such cases, we will give suitable references to the research literature. 

Consider the \textit{coadjoint representation} $\Adj^*:G\rightarrow\GL(\mathfrak{g}^*)$, $g\mapsto\Adj_{g}^*$, defined as follows:
$$(\Adj_g^*(\phi))(x):=\phi(\Adj_{g^{-1}}(x)),\quad g\in G,\text{ }\phi\in\mathfrak{g}^*\text{ },x\in\mathfrak{g}.$$    
This representation gives $\mathfrak{g}^*$ the structure of a $G$-variety, and the resulting orbits are called \textit{coadjoint orbits}. Moreover, $\mathfrak{g}^*$ is known to carry a canonical holomorphic Poisson structure whose symplectic leaves are precisely the coadjoint orbits (see \cite[Sect. 1.2, 1.3]{Chriss}). It follows that every coadjoint orbit inherits a holomorphic symplectic form, namely a closed, non-degenerate, holomorphic $2$-form (see \cite[Sect. 1.1]{Chriss} for further details on symplectic forms). One call this particular symplectic form the \textit{Kirillov-Kostant-Souriau form}.

Now, recall that the Killing form induces a vector space isomorphism $\mathfrak{g}\cong\mathfrak{g}^*$. Noting that the Killing form is $G$-invariant (see \ref{Subsection: The basic objects}), one can verify that our isomorphism is $G$-equivariant. Hence, in addition to identifying $\mathfrak{g}$ with $\mathfrak{g}^*$, the isomorphism identifies adjoint orbits with coadjoint orbits. We may therefore transfer the above-mentioned geometric structures on $\mathfrak{g}^*$ and coadjoint orbits to $\mathfrak{g}$ and adjoint orbits, respectively. The Lie algebra $\mathfrak{g}$ then has a canonical holomorphic Poisson structure and its symplectic leaves are exactly the adjoint orbits. In particular, every adjoint orbit $\mathcal{O}\subseteq\mathfrak{g}$ carries a preferred holomorphic symplectic form, $\omega_{\mathcal{O}}\in\Omega^2(\mathcal{O})$. To describe it, fix $x\in\mathcal{O}$ and identify the tangent space $T_x\mathcal{O}$ with $\mathfrak{g}/C_{\mathfrak{g}}(x)$ via \eqref{Equation: Tangent space description}, so that the restriction of $\omega_{\mathcal{O}}$ to $x$ is a bilinear form on $\mathfrak{g}/C_{\mathfrak{g}}(x)$. We then have
\begin{equation}\label{Equation: The canonical symplectic form}
\omega_{\mathcal{O}}\vert_x:\mathfrak{g}/C_{\mathfrak{g}}(x)\otimes\mathfrak{g}/C_{\mathfrak{g}}(x)\rightarrow\mathbb{C},\quad\omega_{\mathcal{O}}\vert_x([y],[z]):=\langle x,[y,z]\rangle\end{equation}
for all $[y],[z]\in\mathfrak{g}/C_{\mathfrak{g}}(x)$.

Thus far, our treatment of adjoint orbits and geometry has been classical. However, there is a more modern and elaborate framework within which to appreciate the canonical symplectic forms described above -- \textit{hyperk\"{a}hler geometry}. It turns out that each adjoint orbit is a hyperk\"{a}hler manifold (see \cite{Dancer} for a definition) whose underlying holomorphic symplectic form is the canonical one we have just described. This result was proved by Kronheimer \cite{KronheimerInstantons,KronheimerSemisimple} in important special cases and by Biquard \cite{Biquard} and Kovalev \cite{Kovalev} in full generality. The key is to identify each adjoint orbit with a certain moduli space of solutions to \textit{N\"{a}hm's equations} (see \cite[Sect. 2.2]{Bielawski}), known to be a hyperk\"{a}hler manifold, and then transfer the hyperk\"{a}hler structure over to the orbit. It is nevertheless often difficult to describe this induced hyperk\"{a}hler structure in explicit terms (see \cite[Sect 2.3]{Bielawski}). This is one of several unresolved issues that have made and continue to make the hyperk\"{a}hler geometry of adjoint orbits an active area of research (see \cite{BielawskiReducible,Santa-Cruz,Villumsen,Kobak,BielawskiOnthe}).  
 
\section{Semisimple Orbits}\label{Section: Semisimple orbits} 
Informally speaking, Example \ref{Example: Type A adjoint orbits} shows adjoint orbits to be generalizations of matrix conjugacy classes. In what follows, we will study the adjoint orbits that generalize the conjugacy classes of diagonalizable matrices -- the \textit{semisimple orbits}.  

\subsection{Definitions and characterizations}\label{Subsection: Definitions and characterizations}
Recall from \ref{Subsection: semisimple and nilpotent elements} that the subvariety of semisimple elements $\mathfrak{g}_{\text{ss}}\subseteq\mathfrak{g}$ is invariant under the adjoint representation of $G$. One may rephrase this $G$-invariance in the following way: an adjoint orbit contains a semisimple element if and only if it consists of semisimple elements. The adjoint orbits satisfying these equivalent conditions are called the \textit{semisimple orbits}.   

\begin{definition}\label{Definition: Semisimple orbit}
An adjoint orbit $\mathcal{O}\subseteq\mathfrak{g}$ is called a \textit{semisimple orbit} if $\mathcal{O}\cap\mathfrak{g}_{\text{ss}}\neq\emptyset$, or equivalently $\mathcal{O}\subseteq\mathfrak{g}_{\text{ss}}$.
\end{definition}

\begin{example}\label{Example: Type A semisimple orbits} 
Let $G=\SL_n(\mathbb{C})$. Using Example \ref{Example: Type A adjoint orbits} and Theorem \ref{Theorem: equivalence of semisimple / nilpotent}, we see that the semisimple orbits in $\mathfrak{sl}_n(\mathbb{C})$ are precisely the conjugacy classes of semisimple (ie. diagonalizable) $n\times n$ matrices having zero trace. 
\end{example}

While we shall take Definition \ref{Definition: Semisimple orbit} to be our definition of a semisimple orbit, there are alternatives. To motivate one of these, recall that Example \ref{Example: Conjugation} shows the conjugacy class of a diagonal $n\times n$ matrix with pairwise distinct eigenvalues to be closed in $\text{Mat}_{n\times n}(\mathbb{C})$. Now suppose that the matrix in question lies in $\mathfrak{sl}_n(\mathbb{C})$. Its $\SL_n(\mathbb{C})$-adjoint orbit is semisimple (see Example \ref{Example: Type A semisimple orbits}), and the argument from Example \ref{Example: Conjugation} shows the orbit to be closed in $\mathfrak{sl}_n(\mathbb{C})$. At the same time, the matrices $N_1$ and $N_2$ from Example \ref{Example: Conjugation} belong to $\mathfrak{sl}_n(\mathbb{C})$. Their $\SL_n(\mathbb{C})$-adjoint orbits are not semisimple (see Example \ref{Example: Type A semisimple orbits} again) and the arguments from Example \ref{Example: Conjugation} show these orbits to be non-closed in $\mathfrak{sl}_n(\mathbb{C})$. Our findings here are an instance of a general fact: an adjoint orbit is closed if and only if it is semisimple. Before proving this, we present the following lemma.

\begin{lemma}\label{Lemma: Unique semisimple orbit in closure}
	If $\mathcal{O}\subseteq\mathfrak{g}$ is an adjoint orbit, then there exists a unique semisimple orbit belonging to $\overline{\mathcal{O}}$.
\end{lemma}

\begin{proof}
	Choose a point $x\in\mathcal{O}$. For existence, Corollary \ref{Corollary: Semisimple part in closure} implies that $x_s\in\overline{\mathcal{O}}$. It follows that $\mathcal{O}(x_s)$ is a semisimple orbit lying in $\overline{\mathcal{O}}$. As for uniqueness, suppose that $f\in\mathbb{C}[\mathfrak{g}]^G$. Since $f(\Adj_g(x))=f(x)$ for all $g\in G$, we see that $f$ is constant-valued on $\mathcal{O}$. By continuity, $f$ is actually constant-valued on $\overline{\mathcal{O}}$. In particular, any two semisimple $y,z\in\overline{\mathcal{O}}$ must satisfy $f(y)=f(z)$. By Proposition \ref{Proposition: Invariant functions have constant values}, $\mathcal{O}(y)=\mathcal{O}(z)$.  
\end{proof}  

\begin{theorem}\label{Theorem: Semisimple equivalent to closed}
	An adjoint orbit $\mathcal{O}\subseteq\mathfrak{g}$ is semisimple if and only if it is closed in $\mathfrak{g}$.
\end{theorem}

\begin{proof}
	Assume that $\mathcal{O}$ is semisimple and let $y\in\overline{\mathcal{O}}$ be given. It follows that $\mathcal{O}(y)\subseteq\overline{\mathcal{O}}$, or equivalently $\overline{\mathcal{O}(y)}\subseteq\overline{\mathcal{O}}$. Using this together with Corollary \ref{Corollary: Semisimple part in closure}, we see that  $y_s\in\overline{\mathcal{O}}$. The orbit $\mathcal{O}(y_s)$ is then a semisimple orbit belonging to $\overline{\mathcal{O}}$, and Lemma \ref{Lemma: Unique semisimple orbit in closure} implies that $\mathcal{O}=\mathcal{O}(y_s)$. We already know that $\mathcal{O}(y)\subseteq\overline{\mathcal{O}}$ (ie. $\mathcal{O}(y)\leq\mathcal{O}$), while Corollary \ref{Corollary: Semisimple part in closure} gives $\mathcal{O}(y_s)\subseteq\overline{\mathcal{O}(y)}$ (ie. $\mathcal{O}=\mathcal{O}(y_s)\leq\mathcal{O}(y)$). Hence $\mathcal{O}(y)=\mathcal{O}$. This implies that $y\in\mathcal{O}$, and we conclude that $\mathcal{O}$ is closed.
	
	Conversely, assume that $\mathcal{O}$ is closed in $\mathfrak{g}$ and choose a point $x\in\mathcal{O}$. Corollary \ref{Corollary: Semisimple part in closure} implies that $x_s\in\overline{\mathcal{O}}=\mathcal{O}$, meaning that $\mathcal{O}=\mathcal{O}(x_s)$ is a semisimple orbit. 
\end{proof}

\begin{corollary}\label{Corollary: Semisimple orbits are affine}
	If $\mathcal{O}\subseteq\mathfrak{g}$ is a semisimple orbit, then $\mathcal{O}$ is an affine variety.
\end{corollary}

\begin{proof}
	This follows from the fact that $\mathcal{O}$ is closed in the affine variety $\mathfrak{g}$.	
\end{proof}

\subsection{Stabilizer descriptions}\label{Subsection: Stabilizer descriptions}
Recall from \ref{Subsection: semisimple and nilpotent elements} that $x\in\mathfrak{g}$ is semisimple if and only if $\Adj_g(x)\in\mathfrak{t}$ for some $g\in G$. It follows that an adjoint orbit is semisimple if and only if it is of the form $\mathcal{O}(h)$ for $h\in\mathfrak{t}$. Now given $h_1,h_2\in\mathfrak{t}$, one can show that $\mathcal{O}(h_1)=\mathcal{O}(h_2)$ if and only if $h_1$ and $h_2$ belong to the same $W$-orbit in $\mathfrak{t}$. We conclude that $[h]\mapsto \mathcal{O}(h)$ defines a bijection from $\mathfrak{t}/W$ to the set of semisimple orbits. At the same time, each equivalence class in $\mathfrak{t}/W$ has a unique representative belonging to \begin{equation}\label{Equation: Fundamental domain}\mathcal{D}:=\{h\in\mathfrak{t}:\forall\alpha\in\Pi,\text{ Re}(\alpha(h))\geq 0\text{ and }\text{Re}(\alpha(h))=0\Longrightarrow\text{Im}(\alpha(h))\geq 0\}
\end{equation}
(see \cite{Collingwood}, Section 2.2). Our parametrization of semisimple orbits therefore takes the form $h\mapsto\mathcal{O}(h)$, $h\in\mathcal{D}$.

Given $h\in\mathcal{D}$, we have the following instance of the orbit-stabilizer isomorphism \eqref{Equation: Orbit-stabilizer isomorphism}:
\begin{equation}\label{Equation: semisimple orbit-stabilizer isomorphism}
\mathcal{O}(h)\cong G/C_G(h).
\end{equation}
In the interest of using \eqref{Equation: semisimple orbit-stabilizer isomorphism} to study semisimple orbits, we turn our attention to the structure of $C_G(h)$. To this end, set
$$\Pi(h):=\{\alpha\in\Pi:\alpha(h)=0\}.$$ Recall that $\Pi(h)$ determines a standard parabolic subgroup $P_{\Pi(h)}$ with Levi decomposition $P_{\Pi(h)}=U_{\Pi(h)}\rtimes L_{\Pi(h)}$, as explained in \ref{Subsection: The basic objects}. 

\begin{proposition}\label{Proposition: Stabilizer structure}
If $h\in\mathcal{D}$, then $C_G(h)=L_{\Pi(h)}$.
\end{proposition}

\begin{proof}
To begin, note that $P_{\Pi(h)}$ is connected (see \cite[Thm. 28.4.2]{Tauvel}). It follows that the quotient $P_{\Pi(h)}/U_{\Pi(h)}\cong L_{\Pi(h)}$ is also connected. At the same time, $C_G(h)$ is known to be connected (see \cite[Thm. 2.3.3]{Collingwood}). Proving $C_G(h)=L_{\Pi(h)}$ therefore amounts to establishing the equality of these groups' respective Lie algebras  $C_{\mathfrak{g}}(H)$ and $\mathfrak{l}_{\Pi(h)}$. To this end, as $\mathfrak{t}$ is abelian and $h\in\mathfrak{t}$, we have $\mathfrak{t}\subseteq C_{\mathfrak{g}}(h)$. It follows that $C_{\mathfrak{g}}(h)$ decomposes as a direct sum of $\mathfrak{t}$ and certain root spaces. To determine these root spaces note that $\alpha\in\Delta$ satisfies $\mathfrak{g}_{\alpha}\subseteq C_{\mathfrak{g}}(h)$ if and only if $\alpha(h)=0$. Hence, $$C_{\mathfrak{g}}(h)=\mathfrak{t}\oplus\bigoplus_{\substack{\alpha\in\Delta\\ \alpha(h)=0}}\mathfrak{g}_{\alpha}.$$ Using this together with \eqref{Equation: Lie algebra of Levi factor}, we are reduced to proving that \begin{equation}\label{Equation: Reduced to proving}
	\Delta_{\Pi(h)}=\{\alpha\in\Delta:\alpha(h)=0\}.
	\end{equation}
	The inclusion $\Delta_{\Pi(h)}\subseteq\{\alpha\in\Delta:\alpha(h)=0\}$ clearly holds. Conversely, suppose that $\alpha\in\Delta$ satisfies $\alpha(h)=0$. There exist simple roots $\alpha_1,\ldots,\alpha_k\in\Pi$ and non-zero integers $n_1,\ldots,n_k\in\mathbb{Z}\setminus\{0\}$ such that \begin{equation}\label{Equation: Linear combination of simple roots}\alpha=\sum_{i=1}^kn_i\alpha_i.\end{equation} Since $\alpha(h)=0$, we have \begin{equation}\label{Equation: Root sum}
	0=\sum_{i=1}^k n_i\alpha_i(h). 
	\end{equation} In particular, $0=\sum_{i=1}^k n_i\text{Re}(\alpha_i(h))$. Now, recall from \eqref{Equation: Fundamental domain} that $\text{Re}(\alpha_i(h))\geq 0$ for all $i$, and note that all $n_i$ are strictly positive or all $n_i$ are strictly negative. It follows that $\text{Re}(\alpha_i(h))=0$ for all $i$. Noting that $h\in\mathcal{D}$, it must be true that $\text{Im}(\alpha_i(h))\geq 0$ for all $i$. Also, taking imaginary parts in \eqref{Equation: Root sum} gives
	$0=\sum_{i=1}^k n_i\text{Im}(\alpha_i(h))$.
	Again using the fact that all $n_i$ are strictly positive or all $n_i$ are strictly negative, it follows that $\text{Im}(\alpha_i(h))=0$ for all $i$. We conclude that $\alpha_i(h)=0$ (ie. $\alpha_i\in\Pi(h)$) for all $i$, which by \eqref{Equation: Linear combination of simple roots} implies $\alpha\in\Delta_{\Pi(h)}$. Having shown \eqref{Equation: Reduced to proving} to hold, our proof is complete. 
\end{proof}

Proposition \ref{Proposition: Stabilizer structure} gives rise to the following more uniform description of the stabilizers of semisimple elements.

\begin{corollary}\label{Corollary: Stabilizers of semisimple elements are reductive}
If $x\in\mathfrak{g}$ is semisimple, then $C_G(x)$ is a Levi factor of a parabolic subgroup of $G$.
\end{corollary}

\begin{proof}
Since $x$ is semisimple, $x=\Adj_g(h)$ for some $g\in G$ and $h\in\mathcal{D}$. It follows that $C_G(x)=gC_G(h)g^{-1}$, which by Proposition \ref{Proposition: Stabilizer structure} is precisely $gL_{\Pi(h)}g^{-1}$. This is a Levi factor of the parabolic subgroup $gP_{\Pi(h)}g^{-1}$.
\end{proof}

We conclude this section with a proof of the existence of regular adjoint orbits, as promised in \ref{Subsection: Dimension and regularity}. 

\begin{corollary}\label{Corollary: The existence of regular orbits}
There exist regular adjoint orbits in $\mathfrak{g}$.
\end{corollary}

\begin{proof}
Let $\{\epsilon_{\alpha}\}_{\alpha\in\Pi}$ and $h\in\mathfrak{t}$ be as defined in Example \ref{Example: Kostant triple}. Since $\alpha(h)=2$ for all $\alpha\in\Pi$, we see that $h\in\mathcal{D}$ and $\Pi(h)=\emptyset$. It follows that $\mathfrak{l}_{\Pi(h)}=\mathfrak{t}$ (see \eqref{Equation: Lie algebra of Levi factor}), which together with the connectedness of $L_{\Pi(h)}$ (established in the proof of Proposition \ref{Proposition: Stabilizer structure}) implies that $L_{\Pi(h)}=T$. Now the orbit-stabilizer isomorphism \eqref{Equation: semisimple orbit-stabilizer isomorphism} and Proposition \ref{Proposition: Stabilizer structure} give
$$\dim(\mathcal{O}(h))=\dim(G/T)=\dim(G)-\dim(T)=\dim(G)-\text{rank}(G).$$ In other words, $\mathcal{O}(h)$ is a regular adjoint orbit. 
\end{proof}

\subsection{Fibrations over partial flag varieties}\label{Subsection: Fibrations over partial flag varieties}
Consider a point $h\in\mathcal{D}$, to be regarded as fixed throughout \ref{Subsection: Fibrations over partial flag varieties}, and recall the variety isomorphism $O(h)\cong G/C_G(h)$ in \eqref{Equation: semisimple orbit-stabilizer isomorphism}. Proposition \ref{Proposition: Stabilizer structure} allows us to write the right-hand-side of this isomorphism as $G/L_{\Pi(h)}$, and the inclusion $L_{\Pi(h)}\subseteq P_{\Pi(h)}$ gives rise to a fibration $U_{\Pi(h)}\rightarrow G/L_{\Pi(h)}\rightarrow G/P_{\Pi(h)}$. In other words, we have a canonical fibre bundle 
\begin{equation}\label{Equation: Affine bundle}
U_{\Pi(h)}\rightarrow \mathcal{O}(h)\xrightarrow{\pi_h} G/P_{\Pi(h)},
\end{equation}
where $\pi_h(\text{Ad}_g(h))=[g]\in G/P_{\Pi(h)}$ for all $g\in G$.

The fibre $U_{\Pi(h)}$ is unipotent, so that the exponential map restricts to a variety isomorphism $\mathfrak{u}_{\Pi(h)}\xrightarrow{\cong} U_{\Pi(h)}$ (see \cite[Chapt. VIII, Thm. 1.1]{Hochschild}). It follows that the fibres of \eqref{Equation: Affine bundle} are affine spaces, raising the issue of whether $\mathcal{O}(h)$ is isomorphic to the total space of a vector bundle over $G/P_{\Pi(h)}$. Perhaps surprisingly, one cannot expect an isomorphism on the level of algebraic varieties. The total space of a vector bundle over $G/P_{\Pi(h)}$ contains $G/P_{\Pi(h)}$ as the zero-section, and one knows $G/P_{\Pi(h)}$ to be irreducible (since $G$ is irreducible) and projective (see \cite[Cor. 28.1.4]{Tauvel}). It follows that the zero-section is a closed, irreducible, projective subvariety. However, the only a such subvarieties of the affine variety $\mathcal{O}(h)$ are the singletons. A similar argument precludes the existence of a biholomorphism, as a compact, connected, complex submanifold of the Stein manifold $\mathcal{O}(h)$ is necessarily a point.\footnote{For the definitions and results necessary to make this argument, we refer the reader to \cite[Chapt. V, Sect. 1]{Fritzsche}.} Nevertheless, it turns out that there is a diffeomorphism between $\mathcal{O}(h)$ and (the total space of) $T^*(G/P_{\Pi(h)})$ that respects fibrations over $G/P_{\Pi(h)}$. We devote the balance of \ref{Subsection: Fibrations over partial flag varieties} to an explicit construction of this diffeomorphism. 

\begin{remark}\label{Remark: Connection to hyperkahler geometry}
By invoking well-known facts from hyperk\"{a}hler geometry, one can deduce that $\mathcal{O}(h)$ and $T^*(G/P_{\Pi(h)})$ are diffeomorphic without having to construct a diffeomorphism. To this end, recall that $\mathcal{O}(h)$ is a hyperk\"{a}hler manifold (see \ref{Subsection: Geometric features}). Among other things, it follows that the underlying smooth manifold $\mathcal{O}(h)$ carries a family of complex structures. This family turns out to include the usual complex structure on $\mathcal{O}(h)$ (ie. the one it inherits as a subvariety of $\mathfrak{g}$), as well as one in which $\mathcal{O}(h)$ becomes biholomorphic to $T^*(G/P_{\Pi(h)})$ (see \cite[Thm. 2]{Biquard}). One concludes that $\mathcal{O}(h)$ and $T^*(G/P_{\Pi(h)})$ must be diffeomorphic. While this approach is elegant, it is difficult to extract an explicit diffeomorphism $\mathcal{O}(h)\cong T^*(G/P_{\Pi(h)})$ from the supporting arguments.     
\end{remark}

Let $K\subseteq G$ be a maximal compact subgroup having $K\cap T$ as a maximal real torus. We shall assume the Lie algebra $\mathfrak{k}$ of $K$ to have the form 
\begin{equation}\label{Equation: Lie algebra of maximal compact subgroup}
\mathfrak{k}=(\mathfrak{k}\cap\mathfrak{t})\oplus\left(\bigoplus_{\alpha\in\Delta_{+}}\text{span}_{\mathbb{R}}\{e_{\alpha}-e_{-\alpha}\}\right)\oplus\left(\bigoplus_{\alpha\in\Delta_{+}}\text{span}_{\mathbb{R}}\{i(e_{\alpha}+e_{-\alpha})\}\right),
\end{equation}
where $e_{\alpha}\in\mathfrak{g}_{\alpha}\setminus\{0\}$ and $e_{-\alpha}\in\mathfrak{g}_{-\alpha}\setminus\{0\}$ for each $\alpha\in\Delta_+$ (cf. the proof of Theorem 6.11 in \cite{Knapp}).
Now note that $K$ acts transitively on $G/P_{\Pi(h)}$ by left multiplication (see \cite[Chapt. 5, Sect. 8]{Atiyah}), and that the $K$-stabilizer of the identity coset $[e]\in G/P_{\Pi(h)}$ is $K\cap P_{\Pi(h)}=:K_{\Pi(h)}$. By the orbit-stabilizer isomorphism \eqref{Equation: Orbit-stabilizer isomorphism}, this time for compact Lie group actions, we have a $K$-equivariant diffeomorphism. 
\begin{equation}\label{Equation: Equivariant diffeomorphism}
K/K_{\Pi(h)}\xrightarrow{\cong} G/P_{\Pi(h)}.
\end{equation}

We will later need the following result concerning $K_{\Pi(h)}$.

\begin{lemma}\label{Lemma: K-stabilizer}
We have $K_{\Pi(h)}\subseteq L_{\Pi(h)}$.
\end{lemma}

\begin{proof}
We begin by showing $K_{\Pi(h)}$ to be connected. To this end, consider the fibration $K_{\Pi(h)}\rightarrow K\rightarrow K/K_{\Pi(h)}$. The base is simply-connected by virtue of being diffeomorphic to $G/P_{\Pi(h)}$, a simply-connected space (see \cite[Chapt. 3, Sect. 1, Prop. 7]{Akhiezer}). The total space $K$ is homotopy-equivalent to $G$ (see \cite[Appendix A, Thm. 1.1]{Whitehead}), so that the connectedness of $K$ follows from that of $G$. Using these two observations together with the long exact sequence of homotopy groups associated to our fibration (see \cite[Chapt. IV, Cor. 8.6]{Whitehead}), we see that $K_{\Pi(h)}$ is also connected. It will therefore suffice to prove that $\mathfrak{k}_{\Pi(h)}\subseteq\mathfrak{l}_{\Pi(h)}$, where $\mathfrak{k}_{\Pi(h)}$ and $\mathfrak{l}_{\Pi(h)}$ are the Lie algebras of $K_{\Pi(h)}$ and $L_{\Pi(h)}$, respectively. The former Lie algebra is $\mathfrak{k}\cap\mathfrak{p}_{\Pi(h)}$, which via \eqref{Equation: Lie algebra of standard parabolic} and \eqref{Equation: Lie algebra of maximal compact subgroup} can be shown to coincide with $$(\mathfrak{k}\cap\mathfrak{t})\oplus\left(\bigoplus_{\alpha\in\Delta_{\Pi(h)}^{+}}\text{span}_{\mathbb{R}}\{e_{\alpha}-e_{-\alpha}\}\right)\oplus\left(\bigoplus_{\alpha\in\Delta_{\Pi(h)}^{+}}\text{span}_{\mathbb{R}}\{i(e_{\alpha}+e_{-\alpha})\}\right).$$ Since $\mathfrak{t}\subseteq\mathfrak{l}_{\Pi(h)}$ and $\mathfrak{g}_{\alpha}\oplus\mathfrak{g}_{-\alpha}\subseteq\mathfrak{l}_{\Pi(h)}$ for all $\alpha\in\Delta_{\Pi(h)}^+$ (see \ref{Equation: Lie algebra of Levi factor}), it follows that $\mathfrak{k}_{\Pi(h)}\subseteq\mathfrak{l}_{\Pi(h)}$.
\end{proof}

Let us return to the main discussion. Note that the pullback of $T^*(G/P_{\Pi(h)})$ under \eqref{Equation: Equivariant diffeomorphism} is a smooth complex vector bundle $E\rightarrow K/K_{\Pi(h)}$, so that we have the commutative diagram \begin{equation}\label{Equation: Commutative diagram}
\begin{tikzcd}[column sep=30pt]
E \ar{d} \ar{r}{\cong} & T^*(G/P_{\Pi(h)}) \ar{d} \\
K/K_{\Pi(h)} \ar{r}{\cong} & G/P_{\Pi(h)}
\end{tikzcd}
\end{equation}
One can show $E$ to be a $K$-equivariant vector bundle, meaning that $K$ acts smoothly on $E$ via a lift of the left multiplication action on $K/K_{\Pi(h)}$ (hence $k\in K$ sends the fibre over $x\in K/K_{\Pi(h)}$ to the fibre over $k\cdot x$), and secondly that restricting the action of $k\in K$ to the fibre over $x\in K/K_{\Pi(h)}$ gives a vector space isomorphism with the fibre over $k\cdot x$. Since the identity coset $[e]\in K/K_{\Pi(h)}$ is fixed by $K_{\Pi(h)}$, this second condition implies that the fibre over $[e]$ in any $K$-equivariant vector bundle is a complex $K_{\Pi(h)}$-representation. In fact, this turns out to define a bijective correspondence between isomorphism classes of (smooth) complex $K_{\Pi(h)}$-representations and isomorphism classes of smooth $K$-equivariant complex vector bundles over $K/K_{\Pi(h)}$. The inverse process assigns to each representation $\varphi:K_{\Pi(h)}\rightarrow\GL(V)$ the \textit{associated bundle} $K\times_{K_{\Pi(h)}}V$, defined as follows. The group $K_{\Pi(h)}$ acts freely on $K\times V$ via $g\cdot(k,v)=(kg^{-1},\varphi(g)v)$, $g\in K_{\Pi(h)}$, $k\in K$, $v\in V$. One defines $K\times_{K_{\Pi(h)}}V$ to be the quotient manifold $(K\times V)/K_{\Pi(h)}$, equipped with the projection to $K/K_{\Pi(h)}$ given by $[(k,v)]\mapsto[k]$. The fibres of this projection are naturally complex vector spaces, and the $K$-equivariant structure comes from the left-multiplicative action of $K$ on the first factor of $K\times_{K_{\Pi(h)}}V$.

With the above discussion in mind, it is natural to realize $E$ as an associated bundle. To this end, $U_{\Pi(h)}$ is a normal subgroup of $P_{\Pi(h)}$ by virtue of the former being the unipotent radical of the latter. It follows that $\mathfrak{u}_{\Pi(h)}$ is invariant under $P_{\Pi(h)}$ (hence also the subgroup $K_{\Pi(h)}$), acting through the adjoint representation of $G$. In particular, $\mathfrak{u}_{\Pi(h)}$ is a $K_{\Pi(h)}$-representation.

\begin{lemma}\label{Lemma: The fibre of E over [e]}
There is an isomorphism $E\cong K\times_{K_{\Pi(h)}}\mathfrak{u}_{\Pi(h)}$ of $K$-equivariant vector bundles over $K/K_{\Pi(h)}$.
\end{lemma}

\begin{proof}
Given our discussion of the correspondence between isomorphism classes of $K$-equivariant vector bundles and those of $K_{\Pi(h)}$-representations, it will suffice to prove that $\mathfrak{u}_{\Pi(h)}$ is isomorphic to the fibre of $E$ over $[e]\in K/K_{\Pi(h)}$ as a $K_{\Pi(h)}$-representation. The latter is the dual of the tangent space to $[e]\in G/P_{\Pi(h)}$, which by \eqref{Equation: Tangent space description} (with $\mathcal{O}=G/P_{\Pi(h)}$) is isomorphic to $(\mathfrak{g}/\mathfrak{p}_{\Pi(h)})^*$. One can identify $(\mathfrak{g}/\mathfrak{p}_{\Pi(h)})^*$ with $\{\phi\in\mathfrak{g}^*:\phi\vert_{\mathfrak{p}_{\Pi(h)}}=0\}$, which under our isomorphism $\mathfrak{g}^*\cong\mathfrak{g}$ corresponds to $\{x\in\mathfrak{g}:\langle x,y\rangle=0\text{ for all }y\in\mathfrak{p}_{\Pi(h)}\}$. This last subspace can be seen to coincide with $\mathfrak{u}_{\Pi(h)}$, as desired. We leave it as an exercise for the reader to verify that each of the above-constructed isomorphisms is an isomorphism of $P_{\Pi(h)}$ (hence also $K_{\Pi(h)}$-) representations.    
\end{proof} 

Using Lemma \ref{Lemma: The fibre of E over [e]}, we may present \eqref{Equation: Commutative diagram} as 

\begin{equation}\label{Equation: Second commutative diagram}
\begin{tikzcd}[column sep=30pt]
K\times_{K_{\Pi(h)}}\mathfrak{u}_{\Pi(h)} \ar{d} \ar{r}{\cong} & T^*(G/P_{\Pi(h)}) \ar{d} \\
K/K_{\Pi(h)} \ar{r}{\cong} & G/P_{\Pi(h)}.
\end{tikzcd}
\end{equation} 

\begin{proposition}\label{Proposition: Diffeomorphism with the cotangent bundle}
There is a $K$-equivariant diffeomorphism $T^*(G/P_{\Pi(h)})\xrightarrow{\cong}\mathcal{O}(h)$ for which the following diagram commutes: 
 \begin{equation}\label{Equation: Diffeomorphism with cotangent bundle}
 \begin{tikzcd}[column sep=10pt]
 T^*(G/P_{\Pi(h)}) \ar{rr}{\cong}  \ar{dr} & & \mathcal{O}(h) \ar{dl}{\pi_h} \\
 & G/P_{\Pi(h)} &
 \end{tikzcd}.
 \end{equation}
\end{proposition}

\begin{proof}
We will construct a $K$-equivariant diffeomorphism $K\times_{K_{\Pi(h)}}\mathfrak{u}_{\Pi(h)}\xrightarrow{\cong}\mathcal{O}(h)$ making 
\begin{equation}\label{Equation: Third commutative diagram}
\begin{tikzcd}[column sep=30pt]
K\times_{K_{\Pi(h)}}\mathfrak{u}_{\Pi(h)} \ar{d} \ar{r}{\cong} & \mathcal{O}(h) \ar{d}{\pi_h} \\
K/K_{\Pi(h)} \ar{r}{\cong} & G/P_{\Pi(h)}
\end{tikzcd}
\end{equation}
commute. By reversing the horizontal arrows in \eqref{Equation: Second commutative diagram} and combining with \eqref{Equation: Third commutative diagram}, we will have constructed the desired isomorphism $T^*(G/P_{\Pi(h)})\xrightarrow{\cong}\mathcal{O}(h)$. Accordingly, let \\ $K\times_{K_{\Pi(h)}}(P_{\Pi(h)}/L_{\Pi(h)})$ denote the quotient of $K\times (P_{\Pi(h)}/L_{\Pi(h)})$ by the $K_{\Pi(h)}$-action $g\cdot (k,[p])=(kg^{-1},[gp])$, $g\in K_{\Pi(h)}$, $k\in K$, $p\in P_{\Pi(h)}$. The proof of Lemma 2.4 in \cite{Azad} shows the map
\begin{equation}\label{Equation: Biswas diffeomorphism}
K\times_{{K_{\Pi(h)}}}(P_{\Pi(h)}/L_{\Pi(h)})\rightarrow G/L_{\Pi(h)},\quad [(k,[p])]\mapsto[kp]
\end{equation}
to be a (well-defined) diffeomorphism. In the interest of re-writing \eqref{Equation: Biswas diffeomorphism}, recall that $L_{\Pi(h)}=C_G(h)$ (see Proposition \ref{Proposition: Stabilizer structure}). One therefore has the orbit-stabilizer isomorphism \begin{equation}\label{Equation: Second orbit-stabilizer isomorphism}
G/L_{\Pi(h)}\xrightarrow{\cong}\mathcal{O}(h),\quad [g]\mapsto\Adj_g(h).\end{equation} As for the domain of \eqref{Equation: Biswas diffeomorphism}, note that the Levi factorization \eqref{Equation: Levi decomposition of parabolic} implies $U_{\Pi(h)}\rightarrow P_{\Pi(h)}/L_{\Pi(h)}$, $u\mapsto[u]$, is a variety isomorphism. Identifying $U_{\Pi(h)}$ with $\mathfrak{u}_{\Pi(h)}$ via the exponential map, this isomorphism becomes \begin{equation}\label{Equation: Useful isomorphism}
\mathfrak{u}_{\Pi(h)}\rightarrow P_{\Pi(h)}/L_{\Pi(h)},\quad x\mapsto[\exp(x)].\end{equation} Since $K_{\Pi(h)}\subseteq L_{\Pi(h)}$ (see Lemma \ref{Lemma: K-stabilizer}) and $\exp(\Adj_g(x))=g\exp(x)g^{-1}$ for all $g\in K_{\Pi(h)}$ and $x\in\mathfrak{u}_{\Pi(h)}$, one can show \eqref{Equation: Useful isomorphism} to be $K_{\Pi(h)}$-equivariant. This allows us to replace $P_{\Pi(h)}/L_{\Pi(h)}$ in \eqref{Equation: Biswas diffeomorphism} with $\mathfrak{u}_{\Pi(h)}$, giving rise to the diffeomorphism 
\begin{equation}\label{Equation: Second useful isomorphism}
K\times_{K_{\Pi(h)}}\mathfrak{u}_{\Pi(h)}\rightarrow K\times_{K_{\Pi(h)}}(P_{\Pi(h)}/L_{\Pi(h)}),\quad [(k,x)]\mapsto[(k,[\exp(x)])].\end{equation}
Now composing the three diffeomorphisms \eqref{Equation: Biswas diffeomorphism}, \eqref{Equation: Second orbit-stabilizer isomorphism}, and \eqref{Equation: Second useful isomorphism} (in the only order that makes sense), we obtain
\begin{equation}\label{Equation: Final diffeomorphism}
K\times_{K_{\Pi(h)}}\mathfrak{u}_{\Pi(h)}\rightarrow\mathcal{O}(h),\quad [(k,x)]\mapsto\Adj_{k\exp(x)}(h).
\end{equation}
This diffeomorphism is clearly $K$-equivariant. To see that it makes \eqref{Equation: Third commutative diagram} commute, note that the composite map $K\times_{K_{\Pi(h)}}\mathfrak{u}_{\Pi(h)}\rightarrow K/K_{\Pi(h)}\rightarrow G/P_{\Pi(h)}$ sends $[(k,x)]\in K\times_{K_{\Pi(h)}}\mathfrak{u}_{\Pi(h)}$ to $[k]\in G/P_{\Pi(h)}$. On the other hand, the composite map $K\times_{K_{\Pi(h)}}\mathfrak{u}_{\Pi(h)}\rightarrow\mathcal{O}(h)\rightarrow G/P_{\Pi(h)}$ sends $[(k,x)]\in K\times_{K_{\Pi(h)}}\mathfrak{u}_{\Pi(h)}$ to $[k\exp(x)]\in G/P_{\Pi(h)}$. However, since $x\in\mathfrak{u}_{\Pi(h)}\subseteq\mathfrak{p}_{\Pi(h)}$, we must have $\exp(x)\in P_{\Pi(h)}$. It follows that $[k\exp(x)]=[k]$ in $G/P_{\Pi(h)}$, so that \eqref{Equation: Third commutative diagram} indeed commutes. This completes the proof.  
\end{proof}

Proposition \ref{Proposition: Diffeomorphism with the cotangent bundle} has strong implications for the topology (both equivariant and non-equivariant) of a semisimple orbit.

\begin{corollary}\label{Corollary: Equivariant homotopy equivalence}
There is a $K$-equivariant homotopy equivalence between $\mathcal{O}(h)$ and the compact group orbit $\mathcal{O}_K(h)=K\cdot h$.
\end{corollary}

\begin{proof}
Note that $K\times_{K_{\Pi(h)}}\mathfrak{u}_{\Pi(h)}$ is $K$-equivariantly homotopy equivalent to its zero-section. Since \eqref{Equation: Final diffeomorphism} is a $K$-equivariant diffeomorphism, it follows that the image of the zero-section under \eqref{Equation: Final diffeomorphism} is $K$-equivariantly homotopy equivalent to $\mathcal{O}(h)$. The zero-section is given by $\{[(k,0)]\in K\times_{K_{\Pi(h)}}\mathfrak{u}_{\Pi(h)}:k\in K\}$, and its image under \eqref{Equation: Final diffeomorphism} is the $K$-orbit $\mathcal{O}_K(h)\subseteq\mathcal{O}(h)$. This completes the proof.
\end{proof}

An essentially analogous version of Corollary \ref{Corollary: Equivariant homotopy equivalence} is as follows. The cotangent bundle projection $T^*(G/P_{\Pi(h)})\rightarrow G/P_{\Pi(h)}$ is a $K$-equivariant (in fact, $G$-equivariant) homotopy equivalence, so that \eqref{Equation: Diffeomorphism with cotangent bundle} shows $\pi_h:\mathcal{O}(h)\rightarrow G/P_{\Pi(h)}$ to be a $K$-equivariant homotopy equivalence. In principle, one can use this fact to compute various topological invariants of $\mathcal{O}(h)$. Indeed, since $G/P_{\Pi(h)}$ is simply-connected (as discussed in the proof of Lemma \ref{Lemma: K-stabilizer}), so too is $\mathcal{O}(h)$. Secondly, the cohomology ring $H^*(\mathcal{O}(h);\mathbb{Z})$ is isomorphic to $H^*(G/P_{\Pi(h)};\mathbb{Z})$, which has a well-studied description in terms of Schubert calculus (see \cite{Bernstein}). This result has an equivariant counterpart, namely that the $L$-equivariant cohomology rings $H_L^*(\mathcal{O}(h);\mathbb{Z})$ and $H_L^*(G/P_{\Pi(h)};\mathbb{Z})$ are isomorphic for any closed subgroup $L\subseteq K$ (see \cite{Brion} for details on equivariant cohomology). If one takes $L$ to be the maximal torus $K\cap T$ of $K$, then $H_L^*(G/P_{\Pi(h)};\mathbb{Z})$ is describable via equivariant Schubert calculus (see \cite{Graham,Mihaleca}) and (if one works over $\mathbb{C}$ instead of $\mathbb{Z}$) Goresky-Kottwitz-MacPherson (GKM) theory (see \cite{Guillemin}).     

\subsection{Equivariant projective compactifications}\label{Subsection: Equivariant projective compactifications}
While semisimple orbits are affine (see Corollary \ref{Corollary: Semisimple orbits are affine}), each turns out to admit a $G$-equivariant projective compactification (ie. a projective $G$-variety $X$ and a $G$-variety isomorphism between the orbit and an invariant open dense subvariety of $X$). To properly construct this compactification, however, we will need a few preliminary results. Let $w_0\in W$ denote the longest element of the Weyl group (see \cite[Chapt. 20]{Bump}). It is known that $w_0\cdot\Pi=-\Pi$, so that $S\mapsto S^{\vee}:=-w_0\cdot S$ defines an involution on the collection of subsets $S\subseteq\Pi$. With this in mind, we have the following lemma.

\begin{lemma}\label{Lemma: Root set equality}
	For $S\subseteq\Pi$, we have $w_0\cdot\Delta_{S^{\vee}}^{-}=\Delta_{S}^+$.
\end{lemma}

\begin{proof}
	Let us write $S=\{\alpha_1,\ldots,\alpha_n\}$, so that $S^{\vee}=\{\gamma_1,\ldots,\gamma_n\}$ with $\gamma_i=-w_0\cdot\alpha_i$ for all $i$. Given $\beta\in\Delta_S^+$, we have $\beta=\sum_{i=1}^nk_i\alpha_i$ with $k_1,\ldots,k_n\in\mathbb{Z}_{\geq 0}$. It follows that 
	$$w_0\cdot\beta=\sum_{i=1}^n(-k_i)(-w_0\cdot\alpha_i)=\sum_{i=1}^n(-k_i)\gamma_i.$$ Since $w_0\cdot\beta$ is a root, this implies that $w_0\cdot\beta\in\Delta_{S^{\vee}}^{-}$. Using the fact that $w_0^2=e$, we have $\beta\in w_0\cdot\Delta_{S^{\vee}}^{-}$. 
	
	Conversely, if $\beta\in w_0\cdot\Delta_{S^{\vee}}^{-}$, then $\beta=w_0\cdot\left(\sum_{i=1}^nk_i\gamma_i\right)$ with $k_1,\ldots,k_n\in\mathbb{Z}_{\leq 0}$. We conclude that 
	$$\beta=w_0\cdot\left(\sum_{i=1}^nk_i(-w_0\cdot\alpha_i)\right)=\sum_{i=1}^n(-k_i)\alpha_i.$$
	Noting that $\beta$ is a root, this implies that $\beta\in\Delta_{S}^+$.
\end{proof}

Given $S\subseteq\Pi$, the subset $S^{\vee}\subseteq\Pi$ determines a standard parabolic subgroup $P_{S^{\vee}}\subseteq G$ (as discussed in \ref{Subsection: The basic objects}). Now consider the parabolic subgroup $P_S^*\subseteq G$ defined by
\begin{equation}\label{Equation: Dual parabolic}
P_S^*:=w_0P_{S^{\vee}}(w_0)^{-1},
\end{equation}
where $w_0P_{S^{\vee}}(w_0)^{-1}:=gP_{S^{\vee}}g^{-1}$ for any representative $g\in N_G(T)$ of $w_0$.

\begin{proposition}\label{Proposition: Intersection of opposite parabolics}
	For $S\subseteq\Pi$, we have $P_S\cap P_{S}^*=L_S$.
\end{proposition}

\begin{proof}
	Since $L_S$ and $P_S\cap P_{S}^*$ are connected (the former by the beginning of the proof of Proposition \ref{Proposition: Stabilizer structure} and the latter by \cite[Prop. 2.1]{Digne}), it suffices to establish the equality of their respective Lie algebras. The latter Lie algebra is precisely $\mathfrak{p}_S\cap w_0\cdot\mathfrak{p}_{S^{\vee}}$, where $w_0\cdot\mathfrak{p}_{S^{\vee}}$ denotes the adjoint action of a lift of $w_0$ to $N_G(T)$. This intersection contains $\mathfrak{t}$ and therefore decomposes as a direct sum of $\mathfrak{t}$ and the root spaces belonging to both $\mathfrak{p}_{S}$ and $w_0\cdot\mathfrak{p}_{S^{\vee}}$. To identify these root spaces, note that \begin{equation}\label{Equation: First root space decomposition}
	\mathfrak{p}_S=\mathfrak{t}\oplus\left(\bigoplus_{\alpha\in\Delta_+}\mathfrak{g}_{\alpha}\right)\oplus\left(\bigoplus_{\alpha\in\Delta_{S}^{-}}\mathfrak{g}_{\alpha}\right)\end{equation} and 
	\begin{equation}\label{Equation: Second root space decomposition}
	w_0\cdot\mathfrak{p}_{S^{\vee}}=\mathfrak{t}\oplus\left(\bigoplus_{\alpha\in w_0\cdot\Delta_{+}}\mathfrak{g}_{\alpha}\right)\oplus\left(\bigoplus_{\alpha\in w_0\cdot\Delta_{S^{\vee}}^{-}}\mathfrak{g}_{\alpha}\right).\end{equation}
	Since $w_0\cdot\Delta_{+}=\Delta_{-}$ (see \cite[Prop. 20.14]{Bump} and $w_0\cdot\Delta_{S^{\vee}}^{-}=\Delta_{S}^+$ (by Lemma \ref{Lemma: Root set equality}), we may re-write \eqref{Equation: Second root space decomposition} to obtain
	\begin{equation}\label{Equation: Third root space decomposition}
	w_0\cdot\mathfrak{p}_{S^{\vee}}=\mathfrak{t}\oplus\left(\bigoplus_{\alpha\in \Delta_{S}^+}\mathfrak{g}_{\alpha}\right)\oplus\left(\bigoplus_{\alpha\in\Delta_{-}}\mathfrak{g}_{\alpha}\right).
	\end{equation}
	Using \eqref{Equation: First root space decomposition} and \eqref{Equation: Third root space decomposition}, we conclude that 
	$$\mathfrak{p}_{S}\cap w_0\cdot\mathfrak{p}_{S^{\vee}}=\mathfrak{t}\oplus\left(\bigoplus_{\alpha\in\Delta_S^+}\mathfrak{g}_{\alpha}\right)\oplus\left(\bigoplus_{\alpha\in\Delta_S^-}\mathfrak{g}_{\alpha}\right)=\mathfrak{l}_S,$$ the Lie algebra of $L_S$. This completes the proof.
\end{proof}

Returning to the discussion of equivariant projective orbit compactifications, fix $h\in\mathcal{D}$ and consider its semisimple orbit $\mathcal{O}(h)$. Also, let $X$ be any $G$-variety. By means of \eqref{Equation: Second orbit-stabilizer isomorphism}, a $G$-variety isomorphism between $\mathcal{O}(h)$ and an invariant subvariety $Y\subseteq X$ is equivalent to a $G$-variety isomorphism between $G/L_{\Pi(h)}$ and $Y$. Now by the general orbit-stabilizer theory developed in \ref{Subsection: Orbits as quotient varieties}, specifying the latter isomorphism is equivalent to specifying a point $x\in X$ with $C_G(x)=L_{\Pi(h)}$. In this case, $G/L_{\Pi(h)}$ (hence $\mathcal{O}(h)$) is equivariantly identified with the $G$-orbit of $x$ in $X$. 

In light of the above, we seek a projective $G$-variety $X$ and a point $x\in X$ with $C_G(x)=L_{\Pi(h)}$. To this end, consider the diagonal $G$-action (see Example \ref{Example: Diagonal action}) on the projective variety $X=G/P_{\Pi(h)}\times G/{P_{\Pi(h)}^*}$. Consider also the pair of identity cosets $x=([e],[e])\in X$. The $G$-stabilizer of $x$ is $P_{\Pi(h)}\cap P_{\Pi(h)}^*$, which by Proposition \ref{Proposition: Intersection of opposite parabolics} is precisely $L_{\Pi(h)}$. Appealing to the previous paragraph, the following is a $G$-variety isomorphism between $\mathcal{O}(h)$ and the $G$-orbit $G\cdot([e],[e])\subseteq G/P_{\Pi(h)}\times G/{P_{\Pi(h)}^*}$:
 \begin{equation}\label{Equation: First inclusion}
\mathcal{O}(h)\xrightarrow{\cong}G\cdot([e],[e]),\quad \text{Ad}_g(h)\mapsto ([g],[g]),\quad g\in G.
\end{equation}

\begin{theorem}\label{Theorem: Equivariant projective compactification}
	If $h\in\mathcal{D}$, then \eqref{Equation: First inclusion} makes $G/P_{\Pi(h)}\times G/{P_{\Pi(h)}^{*}}$ a $G$-equivariant projective compactification of $\mathcal{O}(h)$.\footnote{For the precise meaning of ``$G$-equivariant projective compactification'', see the beginning of \ref{Subsection: Equivariant projective compactifications}.}
\end{theorem}

\begin{proof}
It remains only to prove that $G\cdot([e],[e])$ is open and dense in $G/P_{\Pi(h)}\times G/{P_{\Pi(h)}^{*}}$. However, since $G\cdot([e],[e])$ is open in its closure (see \ref{Subsection: Generalities on orbits}), it will suffice to prove that $\overline{G\cdot([e],[e])}=G/P_{\Pi(h)}\times G/{P_{\Pi(h)}^{*}}$. Now note that $G/P_{\Pi(h)}\times G/{P_{\Pi(h)}^{*}}$ is irreducible, so that it has no proper irreducible closed subvarieties of codimension zero (see \cite[Sect. 3.2]{Humphreys}). We are therefore reduced to proving that $\overline{G\cdot([e],[e])}$ and $G/P_{\Pi(h)}\times G/{P_{\Pi(h)}^{*}}$ have equal dimensions. The dimension of the former equals that of its open subvariety $G\cdot([e],[e])$, which by \eqref{Equation: First inclusion} equals $\dim(\mathcal{O}(h))$. Proposition \ref{Proposition: Diffeomorphism with the cotangent bundle} implies $\dim(\mathcal{O}(h))=2\dim(G/P_{\Pi(h)})$, so that we have \begin{equation}\label{Equation: First dimension}\dim(\overline{G\cdot([e],[e])})=2\dim(G/P_{\Pi(h)}).\end{equation} At the same time, we have \begin{equation}\label{Equation: Second dimension}
\dim(G/P_{\Pi(h)}\times G/{P_{\Pi(h)}^{*}})=\dim(G/P_{\Pi(h)})+\dim(G/{P_{\Pi(h)}^{*}}).\end{equation} In light of \eqref{Equation: First dimension} and \eqref{Equation: Second dimension}, it will be enough to prove that $\dim(G/P_{\Pi(h)})=\dim(G/{P_{\Pi(h)}^*})$ (or equivalently, $\dim(P_{\Pi(h)})=\dim(P_{\Pi(h)}^*)$). To this end, Lemma \ref{Lemma: Root set equality} implies that $\Delta_{\Pi(h)}^+$ and $\Delta_{\Pi(h)^{\vee}}^{-}$ have the same cardinality. Of course, since $\Delta_{\Pi(h)}^{-}=-\Delta_{\Pi(h)}^+$, this is equivalent to $\Delta_{\Pi(h)}^{-}$ and $\Delta_{\Pi(h)^{\vee}}^-$ having the same cardinality. Now \eqref{Equation: Lie algebra of standard parabolic} implies that $\dim(\mathfrak{p}_{\Pi(h)})=\dim(\mathfrak{p}_{\Pi(h)^{\vee}})$, which one can re-write as $\dim(P_{\Pi(h)})=\dim(P_{\Pi(h)^{\vee}})$. Finally, since $P_{\Pi(h)}^{*}$ was defined to be a conjugate of $P_{\Pi(h)^{\vee}}$ (see \eqref{Equation: Dual parabolic}), the dimensions of $P_{\Pi(h)^{\vee}}$ and $P_{\Pi(h)}^{*}$ must also agree. Hence $\dim(P_{\Pi(h)})=\dim(P_{\Pi(h)}^*)$, as desired. 
\end{proof}

\begin{remark}\label{Remark: Connection to mirror symmetry}
The compactification in Theorem \ref{Theorem: Equivariant projective compactification} turns out to have some context in the homological mirror symmetry program (see \cite{Givental,Hori} for an introduction to this program).  Vaguely speaking, an instance of homological mirror symmetry is an equivalence between symplecto-geometric data on a smooth manifold (called the $A$-side of the symmetry) and algebro-geometric data on a projective variety (called the $B$-side of the symmetry). In \cite{Ballico}, the authors interpret this as a symmetry between so-called $A$-side Landau-Ginzburg (LG) models and $B$-side LG models (see \cite[Sect. 1]{Ballico} for further details). Building on an existing interpretation of a semisimple orbit $\mathcal{O}(h)$ as an $A$-side LG model (explained in \cite[Sect. 1]{Ballico} and based on \cite[Thm. 2.2]{GGSM}), the authors explain how one would make $G/P_{\Pi(h)}\times G/{P_{\Pi(h)}^{*}}$ into the corresponding $B$-side model.      
\end{remark}

\section{Nilpotent Orbits}\label{Section: Nilpotent orbits}
Recall that we introduced semisimple orbits as generalizing the conjugacy classes of diagonalizable matrices. We now discuss the appropriate generalization of a nilpotent conjugacy class -- a \textit{nilpotent orbit}.    

\subsection{Definitions and first results}\label{Subsection: Definitions and first results}
Recall that the nilpotent cone $\mathcal{N}\subseteq\mathfrak{g}$ is invariant under the adjoint action of $G$ (see \ref{Subsection: semisimple and nilpotent elements}). It follows that $\mathcal{N}$ is a union of those adjoint orbits $\mathcal{O}$ for which $\mathcal{O}\cap\mathcal{N}\neq\emptyset$ (ie. $\mathcal{O}$ contains a nilpotent element). Such adjoint orbits will be called \textit{nilpotent orbits}.

\begin{definition}
An adjoint orbit $\mathcal{O}\subseteq\mathfrak{g}$ is called a \textit{nilpotent orbit} if $\mathcal{O}\cap\mathcal{N}\neq\emptyset$, or equivalently $\mathcal{O}\subseteq\mathcal{N}$.
\end{definition}

While this is generally taken to be the definition of a nilpotent orbit, there are several equivalent descriptions. Indeed, one can rephrase Proposition \ref{Proposition: Characterizations of nilpotent elements} in the language of adjoint orbits as follows.

\begin{proposition}\label{Proposition: Characterizations of nilpotent orbits}
If $\mathcal{O}\subseteq\mathfrak{g}$ is an adjoint orbit, then the following conditions are equivalent.
\begin{itemize}
\item[(i)] $\mathcal{O}$ is a nilpotent orbit.
\item[(ii)] $\mathcal{O}$ is invariant under the dilation action of $\mathbb{C}^*$ on $\mathfrak{g}$.
\item[(iii)] $0\in\overline{\mathcal{O}}$.
\item[(iv)] For all $f\in\mathbb{C}[\mathfrak{g}]^G$, $f$ takes the constant value $f(0)$ on $\mathcal{O}$. 
\end{itemize}
\end{proposition} 

\begin{example}\label{Example: Type A nilpotent orbits}
Let $G=\SL_n(\mathbb{C})$. By Theorem \ref{Theorem: equivalence of semisimple / nilpotent} and Example \ref{Example: Type A adjoint orbits}, the nilpotent orbits in $\mathfrak{sl}_n(\mathbb{C})$ are exactly the conjugacy classes of the nilpotent $n\times n$ matrices (which necessarily have zero trace). These conjugacy classes are in turn indexed by the nilpotent $n\times n$ matrices in Jordan canonical form, modulo permutations of Jordan blocks along the diagonal. Of course, each of these equivalence classes of nilpotent Jordan matrices has a unique representative whose block heights, read from top to bottom, form a non-decreasing sequence. The sequences one obtains in this way are precisely the partitions of $n$, ie. the positive integer sequences $\lambda=(\lambda_1,\lambda_2,\ldots,\lambda_k)$, $1\leq k\leq n$, satisfying $\lambda_1\geq\lambda_2\geq\ldots\geq\lambda_k$ and $\lambda_1+\lambda_2+\ldots+\lambda_k=n$. It follows that the nilpotent orbits in $\mathfrak{sl}_n(\mathbb{C})$ are explicitly parametrized by the partitions of $n$. In particular, there are only finitely many nilpotent $\SL_n(\mathbb{C})$-orbits.
\end{example}

This last point about finiteness turns out to hold in much greater generality. More precisely, it turns out there are only finitely many nilpotent $G$-orbits. The proof of this fact, however, is far more complicated than what appears in Example \ref{Example: Type A nilpotent orbits}. One approach uses ideas related to the Springer resolution and Steinberg variety (see \cite[Sect. 3.3]{Chriss} for more details). Another strategy is to classify nilpotent orbits using so-called ``weighted Dynkin diagrams'' and appeal to the fact that there are only finitely many of the latter (see \cite[Chapt. 3]{Collingwood}). Here, we will see the finiteness of nilpotent orbits as a direct consequence of Kostant's work on the adjoint quotient.  

\begin{theorem}\label{Theorem: Finitely many nilpotent orbits}
There are only finitely many nilpotent orbits.
\end{theorem}

\begin{proof}
Recall the description of the adjoint quotient given in \eqref{Equation: The second adjoint quotient}. It will suffice to prove that $\Phi^{-1}(0)=\mathcal{N}$, since Theorem \ref{Theorem: Kostant's theorem} will then show $\mathcal{N}$ to be a union of finitely many adjoint orbits. To this end, Proposition \ref{Proposition: Characterizations of nilpotent elements} implies that $x\in\mathfrak{g}$ is nilpotent if and only if $f(x)=f(0)$ for all $f\in\mathbb{C}[\mathfrak{g}]^G$. Letting $\chi_1,\chi_2,\ldots,\chi_r\in\mathbb{C}[\mathfrak{g}]^G$ be the homogeneous generators mentioned in \ref{Subsection: The adjoint quotient}, this is equivalent to $\chi_j(x)=0$ for all $j=1,\ldots,n$, ie. $\Phi(x)=0$. The preceding argument establishes that $\Phi^{-1}(0)=\mathcal{N}$, as desired. 
\end{proof}

\subsection{The closure order on nilpotent orbits}\label{Subsection: The closure order on nilpotent orbits}
In what follows, we will see that nilpotent orbits are profitably studied in the context of the closure order (see \ref{Subsection: The closure order}) on the set of adjoint $G$-orbits. To help formulate this, recall that a subset $Q$ of a poset is called a ``connected component'' if it is maximal with respect to the following property: if $p,q\in Q$, then there exists a sequence $r_1,\ldots,r_k\in Q$ with $r_1=p$, $r_k=q$ and for all $j\in\{1,\ldots,k-1\}$, $r_j\leq r_{j+1}$ or $r_{j}\geq r_{j+1}$ (ie. $p$ and $q$ are ``connected'' by a sequence of comparisons in $Q$). It turns out that the nilpotent orbits form a ``connected component'' in the poset of adjoint orbits.

\begin{proposition}\label{Proposition: Nilpotent orbits form a connected component}
The nilpotent orbits form a connected component in the poset of adjoint orbits.	
\end{proposition}	

\begin{proof}
To begin, $\{0\}$ is clearly a nilpotent orbit and Proposition \ref{Proposition: Characterizations of nilpotent orbits} (iii) implies that $\{0\}\leq\mathcal{O}$ for all nilpotent orbits $\mathcal{O}$. In other words, all nilpotent orbits are ``connected'' to $\{0\}$ in the above-defined sense. In particular, any two nilpotent orbits are connected by a sequence of comparisons with other nilpotent orbits. It now remains only to show maximality, namely that if a nilpotent orbit $\mathcal{O}$ and an adjoint orbit $\Theta$ satisfy $\Theta\leq\mathcal{O}$ or $\mathcal{O}\leq\Theta$, then $\Theta$ is a nilpotent orbit. To this end, suppose $\mathcal{O}$ and $\Theta$ are nilpotent and adjoint orbits, respectively. If $\Theta\leq\mathcal{O}$, then by definition $\Theta\subseteq\overline{\mathcal{O}}$. Since $\mathcal{N}$ is a closed subvariety of $\mathfrak{g}$ (see \ref{Subsection: semisimple and nilpotent elements}) containing $\mathcal{O}$, we must have $\overline{\mathcal{O}}\subseteq\mathcal{N}$. Hence $\Theta\subseteq\mathcal{N}$, meaning that $\Theta$ is a nilpotent orbit. For the other case, assume that $\mathcal{O}\leq\Theta$. It follows by definition that $\mathcal{O}\subseteq\overline{\Theta}$, and therefore also $\overline{\mathcal{O}}\subseteq\overline{\Theta}$. Since $0\in\overline{\mathcal{O}}$ (see Proposition \ref{Proposition: Characterizations of nilpotent orbits} (iii)), $0\in\overline{\Theta}$ and Proposition \ref{Proposition: Characterizations of nilpotent orbits} shows that     
$\Theta$ is a nilpotent orbit.  
\end{proof}

\begin{remark}
There is no immediate analogue of Proposition \ref{Proposition: Nilpotent orbits form a connected component} for semisimple orbits. For one thing, it is possible for $\mathcal{O}\leq\Theta$ to hold when $\mathcal{O}$ is a semisimple orbit and $\Theta$ is a non-semisimple adjoint orbit. Indeed, let $\Theta$ be any non-semisimple adjoint orbit. Lemma \ref{Lemma: Unique semisimple orbit in closure} implies that there exists a unique semisimple orbit $\mathcal{O}$ satisfying $\mathcal{O}\subseteq\overline{\Theta}$. By construction, $\mathcal{O}\leq\Theta$.    
\end{remark}

\begin{example}\label{Example: Nilpotent orbits and the dominance order}
Let us describe the poset of nilpotent orbits in $\mathfrak{sl}_n(\mathbb{C})$. To this end, recall that Example \ref{Example: Type A nilpotent orbits} uses the partitions of $n$ to explicitly parametrize the nilpotent $\SL_n(\mathbb{C})$-orbits. It follows that the closure order on these nilpotent orbits corresponds to a partial order on the partitions of $n$, which Gerstenhaber \cite[Chapt. 1, Thm. 2]{Gerstenhaber} and Hesselink \cite[Thm. 3.10]{Hesselink} showed to coincide with the \textit{dominance order} on partitions. Two partitions $\lambda=(\lambda_1,\lambda_2,\ldots,\lambda_k)$ and $\mu=(\mu_1,\mu_2,\ldots,\mu_{\ell})$ of $n$ satisfy $\lambda\leq\mu$ in the dominance order precisely when for all $p\in\{1,\ldots,n\}$, $\lambda_1+\lambda_2+\ldots+\lambda_p\leq\mu_1+\mu_2+\ldots+\mu_p$, where we define $\lambda_j:=0$ for $j>k$ and $\mu_{j}:=0$ for $j>\ell$.  
\end{example}

\subsection{The regular nilpotent orbit}\label{Subsection: The regular nilpotent orbit}
Having discussed nilpotent orbits in general, we turn our attention to a few particularly notable orbits. For the first of these, recall that the proof of Theorem \ref{Theorem: Finitely many nilpotent orbits} shows $\mathcal{N}$ to be a fibre of the adjoint quotient. It then follows from Theorem \ref{Theorem: Kostant's theorem} that $\mathcal{N}$ contains a unique regular adjoint orbit. This amounts to the existence of a unique regular nilpotent orbit, which we shall denote by $\mathcal{O}_{\text{reg}}$. Alternatively, one can characterize $\mathcal{O}_{\text{reg}}$ via the closure order as follows.

\begin{proposition}\label{Proposition: Unique maximal}
In the poset of nilpotent orbits with the closure order, $\mathcal{O}_{\emph{reg}}$ is the unique maximal element.
\end{proposition}

\begin{proof}
Let $\mathcal{O}$ be any nilpotent orbit. Since $\mathcal{O}_{\text{reg}}$ is dense in $\mathcal{N}$ (by Theorem \ref{Theorem: Kostant's theorem}), we have $\mathcal{O}\subseteq\mathcal{N}=\overline{\mathcal{O}_{\text{reg}}}$. Hence $\mathcal{O}\leq\mathcal{O}_{\text{reg}}$, as desired.  
\end{proof}

\begin{example}
Recall that Example \ref{Example: Type A nilpotent orbits} used partitions of $n$ to index nilpotent $\SL_n(\mathbb{C})$-orbits, and Example \ref{Example: Nilpotent orbits and the dominance order} in turn mentioned the correspondence between the dominance order on partitions of $n$ and the closure order on nilpotent $\SL_n(\mathbb{C})$-orbits. Furthermore, one readily sees that the partition $\lambda=(n)$ is the unique maximal element in the dominance order. It follows from Proposition \ref{Proposition: Unique maximal} that $\mathcal{O}_{\text{reg}}\subseteq\mathfrak{sl}_n(\mathbb{C})$ is indexed by $\lambda=(n)$, meaning that $\mathcal{O}_{\text{reg}}$ is the conjugacy class of the $n\times n$ nilpotent Jordan block.   
\end{example}

This last example raises the question of whether, in general, $\mathcal{O}_{\text{reg}}$ has any standard representatives. Kostant \cite{Kostant} answered this question in the affirmative, but we will need to recall a few facts about the representation theory of $\mathfrak{sl}_2(\mathbb{C})$ before elaborating on his argument. Let $\phi:\mathfrak{sl}_2(\mathbb{C})\rightarrow\mathfrak{gl}(V)$ be an $\mathfrak{sl}_2(\mathbb{C})$-representation, ie. $V$ is a finite-dimensional complex vector space and $\phi$ is a Lie algebra morphism. One knows such a representation to be completely reducible, meaning that $$V=\bigoplus_{k=1}^n V_k$$ for irreducible $\mathfrak{sl}_2(\mathbb{C})$-subrepresentations $V_k\subseteq V$. Now let $X,H,Y\in\mathfrak{sl}_2(\mathbb{C})$ be the usual generators, as defined in \eqref{Equation: Usual sl2 generators}. Each subrepresentation $V_k$ decomposes into eigenspaces of $\phi(H)$, which are known to be one-dimensional. The eigenvalues of $\phi(H):V_k\rightarrow V_k$ are integers and form a string as follows: if $\lambda_k$ is the largest eigenvalue, then $\lambda_k\geq 0$ and $-\lambda_k,-\lambda_k+2,-\lambda_k+4,\ldots,\lambda_k-2,\lambda_k$ is a complete list of the eigenvalues. Finally, the last facts we will need are the following decompositions of the kernels of $\phi(X):V\rightarrow V$ and $\phi(H):V\rightarrow V$:
\begin{equation}\label{Equation: Kernel decomposition}\ker(\phi(X))=\bigoplus_{k=1}^nV_k(\lambda_k),
\end{equation}
\begin{equation}\label{Equation: Second kernel decomposition}\ker(\phi(H))=\bigoplus_{k=1}^nV_k(0),
\end{equation}
where $V_k(\lambda_k)$ and $V_k(0)$ denote the $\lambda_k$ and $0$-eigenspaces of $\phi(H):V_k\rightarrow V_k$, respectively.\footnote{Note that $0$ need not be an eigenvalue, so that $V_k(0)=\{0\}$ may hold.}   

\begin{proposition}\label{Proposition: Nilpotent orbit representative}
Recalling the notation used in Example \ref{Example: Kostant triple}, the nilpotent element $\xi=\sum_{\alpha\in\Pi}e_{\alpha}$ belongs to $\mathcal{O}_{\emph{reg}}$. 
\end{proposition}        

\begin{proof}
This proof will begin with two intermediate claims. For the first, recall that Example \ref{Example: Kostant triple} includes $\xi$ in an explicit $\mathfrak{sl}_2(\mathbb{C})$-triple $(\xi,h,\eta)$. We claim that every eigenvalue of $\adj_h:\mathfrak{g}\rightarrow\mathfrak{g}$ is an even integer. However, since each summand appearing in the root space decomposition \eqref{Equation: Root space decomposition} of $\mathfrak{g}$ is $\adj_h$-invariant, it will suffice to establish that $\adj_h$ acts on each summand with only even eigenvalues. Now, since $\adj_h(x)=0$ for all $x\in\mathfrak{t}$, we see that $0$ is the unique eigenvalue of $\adj_h:\mathfrak{t}\rightarrow\mathfrak{t}$. For the root space summands, suppose $\beta\in\Delta$. Note that for all $x\in\mathfrak{g}_{\beta}$, we have $\adj_h(x)=\beta(h)x$. To see that the eigenvalue $\beta(h)$ is an even integer, let us write $\beta=\sum_{\alpha\in\Pi}c_{\alpha}\alpha$ for $c_{\alpha}\in\mathbb{Z}$, $\alpha\in\Pi$. Example \ref{Example: Kostant triple} shows that $\alpha(h)=2$ for all $\alpha\in\Pi$, meaning that $\beta(h)=2\sum_{\alpha\in\Pi}c_{\alpha}$ is indeed an even integer. Hence every eigenvalue of $\adj_h:\mathfrak{g}\rightarrow\mathfrak{g}$ is an even integer.

Our second claim is that $C_{\mathfrak{g}}(h)$ and $C_{\mathfrak{g}}(\xi)$ have the same dimension. To see this, we restrict the adjoint representation $\adj:\mathfrak{g}\rightarrow\mathfrak{gl}(\mathfrak{g})$ to a representation of the subalgebra $\mathfrak{a}:=\text{span}\{\xi,h,\eta\}\cong\mathfrak{sl}_2(\mathbb{C})$ on $\mathfrak{g}$. One can now think of $\mathfrak{g}$ as an $\mathfrak{sl}_2(\mathbb{C})$-representation in which $X,H,Y\in\mathfrak{sl}_2(\mathbb{C})$ are identified with $\xi,h,\eta\in\mathfrak{a}$, respectively. Now, decompose $\mathfrak{g}$ into irreducible $\mathfrak{a}$-subrepresentations,
\begin{equation}\label{Equation: Decomposition into sl2 reps}
\mathfrak{g}=\bigoplus_{k=1}^nV_k.
\end{equation}
Since the eigenvalues of $\adj_h:\mathfrak{g}\rightarrow\mathfrak{g}$ are even integers, the same is necessarily true of the eigenvalues of $\adj_h$ on each irreducible subrepresentation $V_k$. The string description of the latter eigenvalues (explained just before this proposition) implies that $0$ occurs as an eigenvalue in each $V_k$. Since $V_k(0)$ is then one-dimensional for each $k$, \eqref{Equation: Second kernel decomposition} implies that $\dim(\ker(\adj_h))$ equals $n$, the number of irreducible summands in \eqref{Equation: Decomposition into sl2 reps}. A similar argument uses \eqref{Equation: Kernel decomposition} to show that $\dim(\ker(\adj_{\xi}))=n$. Of course, as $\ker(\adj_h)=C_{\mathfrak{g}}(h)$ and $\ker(\adj_{\xi})=C_{\mathfrak{g}}(\xi)$, we have actually shown that $\dim(C_{\mathfrak{g}}(h))=\dim(C_{\mathfrak{g}}(\xi))$.

We may now give a direct proof of our proposition, which is equivalent to the statement that $\mathcal{O}(\xi)=\mathcal{O}_{\text{reg}}$. Now, as $\xi$ is nilpotent and $\mathcal{O}_{\text{reg}}$ is the unique regular nilpotent orbit, this is equivalent to $\mathcal{O}(\xi)$ being a regular adjoint orbit. This is in turn equivalent to $\dim(C_G(\xi))=\text{rank}(G)$, since $\dim(\mathcal{O}(\xi))=\dim(G/C_G(\xi))=\dim(G)-\dim(C_G(\xi))$. At the same time, we know $\dim(C_G(\xi))=\dim(C_{\mathfrak{g}}(\xi))=\dim(C_{\mathfrak{g}}(h))=\dim(C_G(h))$. Our task is therefore to show that $\dim(C_{G}(h))=\text{rank}(G)$. However, the proof of Corollary \ref{Corollary: The existence of regular orbits} explains that $C_G(h)=L_{\Pi(h)}=T$ for our choice of $h$, so that the $\dim(C_{G}(h))=\dim(T)=\text{rank}(G)$. This completes the proof.            
\end{proof}

\begin{remark}
Kostant found a curious relationship between the decomposition \eqref{Equation: Decomposition into sl2 reps} and the topology of $G$. To formulate it, note that the last two paragraphs in the proof of Proposition \ref{Proposition: Nilpotent orbit representative} can be used to show that \eqref{Equation: Decomposition into sl2 reps} has exactly $r:=\text{rank}(G)$ irreducible summands. If $d_1,d_2,\ldots,d_r$ are their respective dimensions, then the Poincar\'{e} polynomial of $G$ factors as follows:
\begin{equation}\label{Equation: Poincare polynomial}P_G(t)=\prod_{j=1}^r(1+t^{d_j})
\end{equation}
(see \cite[Cor. 8.7]{Kostant}).\footnote{Strictly speaking, Kostant's result is that \eqref{Equation: Poincare polynomial} is the Poincar\'{e} polynomial of the adjoint group of $\mathfrak{g}$. The adjoint group is just the quotient of $G$ by its (finite) centre, and one can use this fact to show that \eqref{Equation: Poincare polynomial} must also be the Poincar\'{e} polynomial of $G$.}   
\end{remark}

\subsection{The minimal nilpotent orbit}\label{Subsection: The minimal nilpotent orbit}
Having studied $\mathcal{O}_{\text{reg}}$ as the unique maximal element in the poset of nilpotent orbits (see Proposition \ref{Proposition: Unique maximal}), it is natural to wonder about an analogous result at the opposite extreme --- the existence of a unique minimal nilpotent orbit. However, a first investigation yields underwhelming results. The singleton $\{0\}$ is the unique minimal nilpotent orbit, since Proposition \ref{Proposition: Characterizations of nilpotent orbits} implies that $\{0\}\subseteq\overline{\mathcal{O}}$ for all nilpotent orbits $\mathcal{O}$. It is therefore a natural next step to study the minimal non-zero nilpotent orbits. Interestingly, there turns out to be a unique such orbit when $G$ is simple. At the same time, $G$ being simple implies the existence of a unique $\theta\in\Delta$ with the property of being maximal among all roots in the partial order \eqref{Equation: Partial order on the weight lattice} (see \cite[Thm. 18.9.2]{Tauvel}). We then have the following theorem, which is essentially a restatement of Theorem 4.3.3 in \cite{Collingwood}.  

\begin{theorem}\label{Theorem: Existence of a minimal nilpotent orbit}
If $G$ is simple and $e_{\theta}\in\mathfrak{g}_{\theta}\setminus\{0\}$, then $\mathcal{O}(e_{\theta})$ is the unique minimal element in the poset of non-zero nilpotent orbits.
\end{theorem} 

One denotes the above orbit by $\mathcal{O}_{\text{min}}$ and calls it the \textit{minimal nilpotent orbit}. Note that being minimal in the poset of non-zero nilpotent orbits amounts to the condition $\overline{\mathcal{O}_{\text{min}}}=\mathcal{O}_{\text{min}}\cup\{0\}$. 

\begin{example}
Let $G=\SL_n(\mathbb{C})$, a simple group. Recall from Example \ref{Example: Nilpotent orbits and the dominance order} the poset isomorphism between nilpotent $\SL_n(\mathbb{C})$-orbits in the closure order and partitions of $n$ in the dominance order. While $\lambda=(1,1,\ldots,1)$ (which corresponds to the orbit $\{0\}$) is clearly the unique minimal partition of $n$, $\mu=(2,1,1,\ldots,1)$ is seen to be minimal once $\mu$ is excluded. It follows that $\mu$ corresponds to $\mathcal{O}_{\text{min}}\subseteq\mathfrak{sl}_n(\mathbb{C})$, or equivalently that $\mathcal{O}_{\text{min}}$ is the conjugacy class of a nilpotent Jordan canonical form matrix having a single $2\times 2$ Jordan block and $0$ for all other entries.  
\end{example}

Theorem \ref{Theorem: Existence of a minimal nilpotent orbit} assumes $G$ is simple, so one might legitimately ask whether a unique minimal non-zero nilpotent orbit exists when $G$ is semisimple. The following example shows this to be false in general.   

\begin{example}
Assume that $G$ is simple, in which case the product $G\times G$ is semisimple but not simple. The latter group has Lie algebra $\mathfrak{g}\oplus\mathfrak{g}$, and the adjoint representation $\Adj:G\times G\rightarrow\GL(\mathfrak{g}\oplus\mathfrak{g})$ is as follows:
$$\Adj_{(g_1,g_2)}(x_1,x_2)=(\Adj_{g_1}(x_1),\Adj_{g_2}(x_2))$$ for all $(g_1,g_2)\in G\times G$ and $(x_1,x_2)\in\mathfrak{g}\oplus\mathfrak{g}$, where (via a slight abuse of notation) $\Adj_{g_1}(x_1)$ and $\Adj_{g_2}(x_2)$ come from the adjoint representation of $G$. Now, let $\theta\in\Delta$ and $e_{\theta}\in\mathfrak{g}_{\theta}\setminus\{0\}$ be as discussed above. It follows that the adjoint $G\times G$-orbits of $(e_{\theta},0)$ and $(0,e_{\theta})$ are $$\mathcal{O}_{G\times G}(e_{\theta},0)=\mathcal{O}_{\text{min}}\times\{0\}\quad\text{and}\quad\mathcal{O}_{G\times G}(0,e_{\theta})=\{0\}\times\mathcal{O}_{\text{min}},$$
respectively, where $\mathcal{O}_{\text{min}}\subseteq\mathfrak{g}$ is the minimal nilpotent $G$-orbit. Taking closures, we have 
$$\overline{\mathcal{O}_{G\times G}(e_{\theta},0)}=\overline{\mathcal{O}_{\text{min}}}\times\{0\}=(\mathcal{O}_{\text{min}}\cup\{0\})\times\{0\}=\mathcal{O}_{G\times G}(e_{\theta},0)\cup\{(0,0)\}$$ and
$$\overline{\mathcal{O}_{G\times G}(0,e_{\theta})}=\{0\}\times\overline{\mathcal{O}_{\text{min}}}=\{0\}\times(\mathcal{O}_{\text{min}}\cup\{0\})=\{(0,0)\}\cup\mathcal{O}_{G\times G}(0,e_{\theta}).$$ These calculations establish two things. Firstly, since the closures of $\mathcal{O}_{G\times G}(e_{\theta},0)$ and $\mathcal{O}_{G\times G}(0,e_{\theta})$ contain $(0,0)$, each orbit is nilpotent (see Proposition \ref{Proposition: Characterizations of nilpotent orbits}). Secondly, for each orbit, $\{(0,0)\}$ is the unique orbit that is strictly smaller in the closure order. Equivalently, $\mathcal{O}_{G\times G}(e_{\theta},0)$ and $\mathcal{O}_{G\times G}(0,e_{\theta})$ are minimal among the non-zero nilpotent orbits in the closure order. It is nevertheless clear that $\mathcal{O}_{G\times G}(e_{\theta},0)\neq\mathcal{O}_{G\times G}(0,e_{\theta})$.
\end{example}

\subsection{Orbit projectivizations}\label{Subsection: Orbit projectivizations}
Recall from Proposition \ref{Proposition: Characterizations of nilpotent orbits} that one can characterize nilpotent orbits as those adjoint orbits which are invariant under the dilation action of $\mathbb{C}^*$ on $\mathfrak{g}$. We will see that this fact gives rise to a special geometry enjoyed only by the nilpotent orbits. Indeed, given a non-zero nilpotent orbit $\mathcal{O}\subseteq\mathfrak{g}$, we may define its \textit{projectivization} to be the quotient 
$$\mathbb{P}(\mathcal{O}):=\mathcal{O}/\mathbb{C}^*.$$ Note that $\mathbb{P}(\mathcal{O})$ is, by construction, a subset of the projective space $\mathbb{P}(\mathfrak{g})=(\mathfrak{g}\setminus\{0\})/\mathbb{C}^*$. The latter is a $G$-variety with action map
$$G\times\mathbb{P}(\mathfrak{g})\rightarrow\mathbb{P}(\mathfrak{g}),\quad (g,[x])\mapsto[\Adj_g(x)],\quad g\in G, \text{ }[x]\in\mathbb{P}(\mathfrak{g})$$
(cf. Example \ref{Example: Descent to projective space}), and it is not difficult to check that $\mathbb{P}(\mathcal{O})$ is a $G$-orbit in $\mathbb{P}(\mathfrak{g})$. As such, $\mathbb{P}(\mathcal{O})$ is a smooth locally closed subvariety of $\mathbb{P}(\mathfrak{g})$ (see \ref{Subsection: Generalities on orbits}). However, $\mathbb{P}(\mathcal{O})$ is seldom closed.

\begin{proposition}\label{Proposition: Unique closed orbit}
Assume that $G$ is simple and let $\mathcal{O}\subseteq\mathfrak{g}$ be a non-zero nilpotent orbit. Then $\mathbb{P}(\mathcal{O})$ is closed in $\mathbb{P}(\mathfrak{g})$ if and only if $\mathcal{O}=\mathcal{O}_{\emph{min}}$.
\end{proposition}

\begin{proof}
The nilpotent cone $\mathcal{N}\subseteq\mathfrak{g}$ is $\mathbb{C}^*$-invariant, and one may define $$\mathbb{P}(\mathcal{N}):=(\mathcal{N}\setminus\{0\})/\mathbb{C}^*.$$ It is not difficult to see that $\mathbb{P}(\mathcal{N})$ is a $G$-invariant closed subvariety of $\mathbb{P}(\mathfrak{g})$ whose $G$-orbits are precisely the projectivizations of the non-zero nilpotent orbits. It will therefore be equivalent to prove that $\mathbb{P}(\mathcal{O}_{\text{min}})$ is the unique closed orbit in $\mathbb{P}(\mathcal{N})$.\footnote{Strictly speaking, this will only show $\mathbb{P}(\mathcal{O})$ to be closed if and only if $\mathbb{P}(\mathcal{O})=\mathbb{P}(\mathcal{O}_{\text{min}})$. We leave it as an exercise to check that $\mathbb{P}(\mathcal{O})=\mathbb{P}(\mathcal{O}_{\text{min}})$ if and only if $\mathcal{O}=\mathcal{O}_{\text{min}}$.} Now, remember that $\overline{\mathbb{P}(\mathcal{O})}$ is a union of $\mathbb{P}(\mathcal{O})$ and the $G$-orbits lying below $\mathbb{P}(\mathcal{O})$ in the closure order. Showing $\mathbb{P}(\mathcal{O}_{\text{min}})$ to be the unique closed orbit is therefore equivalent to showing it to be the unique minimal orbit in $\mathbb{P}(\mathcal{N})$. We shall do the latter.
 
Let $\mathcal{O}\subseteq\mathfrak{g}$ be any non-zero nilpotent orbit. Since $\mathcal{O}_{\text{min}}$ is the unique minimal non-zero nilpotent orbit in the closure order, it must be true that $\mathcal{O}_{\text{min}}\subseteq\overline{\mathcal{O}}\setminus\{0\}$. Also, because $\mathcal{N}$ is closed in $\mathfrak{g}$ and contains $\mathcal{O}$,  $\overline{\mathcal{O}}\setminus\{0\}$ is the closure of $\mathcal{O}$ in $\mathcal{N}\setminus\{0\}$. The continuity of the quotient map $\pi:\mathcal{N}\setminus\{0\}\rightarrow\mathbb{P}(\mathcal{N})$ therefore implies $\pi(\overline{\mathcal{O}}\setminus\{0\})\subseteq\overline{\pi(\mathcal{O})}$, and we have
$$\mathbb{P}(\mathcal{O}_{\text{min}})=\pi(\mathcal{O}_{\text{min}})\subseteq\pi(\overline{\mathcal{O}}\setminus\{0\})\subseteq\overline{\pi(\mathcal{O})}=\overline{\mathbb{P}(\mathcal{O})}.$$
In other words, $\mathbb{P}(\mathcal{O}_{\text{min}})$ belongs to the closure of $\mathbb{P}(\mathcal{O})$ for all non-zero nilpotent orbits $\mathcal{O}$. We conclude that $\mathbb{P}(\mathcal{O}_{\text{min}})$ is the unique minimal orbit in $\mathbb{P}(\mathcal{N})$, as desired.   
\end{proof}

Seeking to build on Proposition \ref{Proposition: Unique closed orbit}, we will assume $G$ to be simple for the duration of \ref{Subsection: Orbit projectivizations}.

\begin{remark}
Recall that the image of a projective variety under a morphism is necessarily closed. If $\mathbb{P}(\mathcal{O})$ is projective, we may take this morphism to be the inclusion $\mathbb{P}(\mathcal{O})\hookrightarrow\mathbb{P}(\mathfrak{g})$ and conclude that $\mathbb{P}(\mathcal{O})$ is closed in $\mathbb{P}(\mathfrak{g})$. Conversely, if $\mathbb{P}(\mathcal{O})$ is closed in $\mathbb{P}(\mathfrak{g})$, then the former is necessarily projective. These last two sentences explain that $\mathbb{P}(\mathcal{O})$ is projective if and only if it is closed in $\mathbb{P}(\mathfrak{g})$, which by Proposition \ref{Proposition: Unique closed orbit} is equivalent to $\mathcal{O}=\mathcal{O}_{\text{min}}$. In particular, $\mathbb{P}(\mathcal{O})$ is generally not projective.
\end{remark}

While Proposition \ref{Proposition: Unique closed orbit} is an important first step in understanding $\mathbb{P}(\mathcal{O}_{\text{min}})$, one can say a great deal more. Indeed, as $\mathbb{P}(\mathcal{O}_{\text{min}})$ is a $G$-orbit, it must be $G$-equivariantly isomorphic to $G/P$ for some closed subgroup $P\subseteq G$. Furthermore, since $\mathbb{P}(\mathcal{O}_{\text{min}})$ is projective, $P$ is necessarily a parabolic subgroup of $G$ (see \cite[Cor. 28.1.4]{Tauvel}). This subgroup contains a conjugate of $B$, meaning that $P$ is conjugate to a standard parabolic subgroup $P_S$ (as defined in \ref{Subsection: The basic objects}). One can use this to verify that $G/P\cong G/P_S$, so that $\mathbb{P}(\mathcal{O}_{\text{min}})$ and $G/P_S$ are $G$-equivariantly isomorphic for some subset $S\subseteq\Pi$. The following proposition shows that one can take $S$ to be 
$$\Pi(\theta):=\{\alpha\in\Pi:\langle\alpha,\theta\rangle=0\},$$ the set of simple roots that are orthogonal to the maximal root $\theta$.  

\begin{proposition}\label{Proposition: Description of P(Omin) as a partial flag variety}
	If $G$ is simple, then there is a $G$-variety isomorphism $G/P_{\Pi(\theta)}\cong\mathbb{P}(\mathcal{O}_{\emph{min}})$.
\end{proposition}

\begin{proof}
Fix $e_{\theta}\in\mathfrak{g}_{\theta}\setminus\{0\}$ and recall that $e_{\theta}\in\mathcal{O}_{\text{min}}$ (see Theorem \ref{Theorem: Existence of a minimal nilpotent orbit}). It follows that $[e_{\theta}]\in\mathbb{P}(\mathcal{O}_{\text{min}})$, so that $\mathbb{P}(\mathcal{O}_{\text{min}})$ is the $G$-orbit of $[e_{\theta}]$ in $\mathbb{P}(\mathfrak{g})$. We claim that $C_G([e_{\theta}])=P_{\Pi(\theta)}$, which together with the orbit-stabilizer isomorphism \eqref{Equation: Orbit-stabilizer isomorphism} will prove the proposition. 

Since $G$ is simple, its adjoint representation is irreducible (see \cite[Sect. 0.11]{HumphreysConjugacy}) and $\theta$ is the maximal weight of this representation. It follows that $\Adj(b)(e_{\theta})$ is a non-zero multiple of $e_{\theta}$ for each $b\in B$ (see \cite[Sect. 31.3]{Humphreys}). This shows that $B\subseteq C_G([e_{\theta}])$, meaning that $C_G([e_{\theta}])$ is a standard parabolic subgroup of $G$. In other words, $C_G([e_{\theta}])=P_S$ for some subset $S\subseteq\Pi$, and we must show $S=\Pi(\theta)$. Now by the correspondence between subsets of $\Pi$ and standard parabolic subgroups (see \ref{Subsection: The basic objects}), $S$ is the set of simple roots $\alpha$ for which $\mathfrak{g}_{-\alpha}$ belongs to $C_{\mathfrak{g}}([e_{\theta}])$, the Lie algebra of $C_G([e_{\theta}])$. One can show that \\
$C_{\mathfrak{g}}([e_{\theta}])=\{\xi\in\mathfrak{g}:[\xi,\mathfrak{g}_{\theta}]\subseteq \mathfrak{g}_{\theta}\}$, so that $\alpha\in\Pi$ belongs to $S$ if and only if $[\mathfrak{g}_{-\alpha},\mathfrak{g}_{\theta}]\subseteq\mathfrak{g}_{\theta}$. Furthermore, as $[\mathfrak{g}_{-\alpha},\mathfrak{g}_{\theta}]\cap\mathfrak{g}_{\theta}=\{0\}$, it follows that $\alpha$ belongs to $S$ if and only if $[\mathfrak{g}_{-\alpha},\mathfrak{g}_{\theta}]=\{0\}$. Hence, showing $S=\Pi(\theta)$ amounts to proving that $\alpha\in\Pi$ satisfies $[\mathfrak{g}_{-\alpha},\mathfrak{g}_{\theta}]=\{0\}$ if and only if $\langle\alpha,\theta\rangle=0$. 
	
Assume that $[\mathfrak{g}_{-\alpha},\mathfrak{g}_{\theta}]=\{0\}$.
Since $\theta$ is the maximal root, $\alpha+\theta$ is not a root and we have 
$[\mathfrak{g}_{\alpha},\mathfrak{g}_{\theta}]\subseteq\mathfrak{g}_{\alpha+\theta}=\{0\}$. Applying the Jacobi identity, we see that
\begin{equation}\label{Equation: Three Root Spaces}
[[\mathfrak{g}_{-\alpha},\mathfrak{g}_{\alpha}],\mathfrak{g}_{\theta}]=\{0\}.
\end{equation} Letting $h_{\alpha}\in[\mathfrak{g}_{-\alpha},\mathfrak{g}_{\alpha}]$ be the coroot associated with $\alpha$, \eqref{Equation: Three Root Spaces} gives 
$$0=[h_{\alpha},e_{\theta}]=\theta(h_{\alpha})e_{\theta}=2\frac{\langle\alpha,\theta\rangle}{\langle\alpha,\alpha\rangle}e_{\theta}.$$ It follows that $\langle\alpha,\theta\rangle=0$.
	
Conversely, assume that $\langle\alpha,\theta\rangle=0$. If $[\mathfrak{g}_{-\alpha},\mathfrak{g}_{\theta}]\neq\{0\}$, then $\mathfrak{g}_{\theta-\alpha}\neq\{0\}$ and $\theta-\alpha$ is a root. Since the set of roots is invariant under the action of $W$ on $\mathfrak{t}^*$, $s_{\alpha}\cdot(\theta-\alpha)$ is also a root. However, $$s_{\alpha}\cdot(\theta-\alpha)=s_{\alpha}\cdot\theta-s_{\alpha}\cdot\alpha=\theta-2\frac{\langle\alpha,\theta\rangle}{\langle\alpha,\alpha\rangle}\alpha+\alpha=\theta+\alpha,$$ and this contradicts the maximality of $\theta$. We conclude that $[\mathfrak{g}_{-\alpha},\mathfrak{g}_{\theta}]=\{0\}$. 
\end{proof}

\begin{example}\label{Example: Description of P(Omin) in type A}
One can use Proposition \ref{Proposition: Description of P(Omin) as a partial flag variety} to give a more classical description of $\mathbb{P}(\mathcal{O}_{\text{min}})$ when $G=\SL_n(\mathbb{C})$. To do this,  recall the $\SL_n(\mathbb{C})$-specific discussion of the maximal torus, Borel subgroup, roots, simple roots, etc. given in Example \ref{Example: The type A setup}. It turns out that $\theta=t_1-t_n$ is the maximal root, which by the Killing form description \eqref{Equation: Type A Killing form} implies that \\ $\Pi(\theta)=\{t_i-t_{i+1}:2\leq i\leq n-2\}$. The set $\Delta_{\Pi(\theta)}^{-}$ of negative roots expressible as linear combinations of the roots in $\Pi(\theta)$ is $\Delta_{\Pi(\theta)}^{-}=\{t_i-t_j:2\leq j<i\leq n-1\}$. Using \eqref{Equation: Lie algebra of standard parabolic} to compute $\mathfrak{p}_{\Pi(\theta)}$, one finds the corresponding parabolic subgroup $P_{\Pi(\theta)}\subseteq\SL_n(\mathbb{C})$ to be
\begin{equation}\label{Equation: Type A parabolic for P(Omin)}P_{\Pi(\theta)}=\{(A_{ij})\in\SL_n(\mathbb{C}):A_{i1}=0\text{ for all }i\geq 2\text{ and }A_{nj}=0\text{ for all }j\leq n-1\}.
\end{equation}
In other words, a matrix $A\in\SL_n(\mathbb{C})$ belongs to $P_{\Pi(\theta)}$ if and only if it is block upper-triangular with blocks of dimensions $1\times 1$, $(n-2)\times (n-2)$, and $1\times 1$ along the diagonal (read from top to bottom). 

Now, recall that Example \ref{Example: Stabilizer in a flag variety} computed the $\GL_n(\mathbb{C})$-stabilizer of the flag $V_{\bullet}^e\in\text{Flag}(d_{\bullet};\mathbb{C}^n)$. Note that the $\SL_n(\mathbb{C})$-stabilizer of $V_{\bullet}^e$ is obtained by intersecting $\SL_n(\mathbb{C})$ with the $\GL_n(\mathbb{C})$-stabilizer. Furthermore, when $d_{\bullet}=(1,n-1)$, it is not difficult to see that $C_{\SL_n(\mathbb{C})}(V_{\bullet}^e)$ coincides with \eqref{Equation: Type A parabolic for P(Omin)}. The action of $\SL_n(\mathbb{C})$ on $\text{Flag}(1,n-1;\mathbb{C}^n)$ is transitive, so that we have
$\text{Flag}(1,n-1;\mathbb{C}^n)\cong\SL_n(\mathbb{C})/C_{\SL_n(\mathbb{C})}(V_{\bullet}^e)=\SL_n(\mathbb{C})/P_{\Pi(\theta)}$. Proposition \eqref{Proposition: Description of P(Omin) as a partial flag variety} identifies $\SL_n(\mathbb{C})/P_{\Pi(\theta)}$ with $\mathbb{P}(\mathcal{O}_{\text{min}})$, meaning that we have an $\SL_n(\mathbb{C})$-variety isomorphism
$$\mathbb{P}(\mathcal{O}_{\text{min}})\cong\text{Flag}(1,n-1;\mathbb{C}^n).$$
\end{example}

\begin{remark}
Using similar ideas, one can obtain classical descriptions of $\mathbb{P}(\mathcal{O}_{\text{min}})$ in other Lie types. In type $D_n$, for instance, $\mathbb{P}(\mathcal{O}_{\text{min}})$ identifies with the Grassmannian of isotropic $2$-dimensional subspaces in $\mathbb{C}^{2n}$ (see \cite[Sect. 4.2]{CrooksRayan}) 
\end{remark}

On a somewhat different note, $\mathbb{P}(\mathcal{O}_{\text{min}})$ turns out to feature prominently in \textit{quaternionic-K\"{a}hler geometry} (see \cite{Swann} for a brief introduction). To develop this point, let $X$ be a complex manifold and $E$ a holomorphic subbundle of $TX$. One then has a short exact sequence of holomorphic bundles on $X$,
$$0\rightarrow E\rightarrow TX\xrightarrow{\theta} F\rightarrow 0,$$ where $F:=TX/E$ and $\theta$ is the quotient map. Let $\mathcal{E}$ and $\mathcal{F}$ denote the sheaves of sections of $E$ and $F$, respectively, and consider the following pairing:
\begin{equation}\label{Equation: Pairing}\mathcal{E}\otimes \mathcal{E}\rightarrow \mathcal{F},\quad(s_1,s_2)\mapsto \theta([s_1,s_2]).\footnote{Note that $s_1$ and $s_2$ may be regarded as vector fields via the inclusion $E\subseteq TX$, so that $[s_1,s_2]$ makes sense.}\end{equation}
One calls $E$ a \textit{complex contact structure} if $E$ has co-rank one in $TX$ and the pairing \eqref{Equation: Pairing} is non-degenerate (on both local and global sections of $E$). Once endowed with a complex contact structure, $X$ is called a \textit{complex contact manifold}.

Now, let $\mathcal{O}\subseteq\mathfrak{g}$ be a non-zero nilpotent orbit and recall the canonical symplectic form $\omega_{\mathcal{O}}\in\Omega^2(\mathcal{O})$ defined in \eqref{Equation: The canonical symplectic form}. One can use $\omega_{\mathcal{O}}$ to induce a canonical complex contact structure on $\mathbb{P}(\mathcal{O})$ (see \cite[Prop. 2.2, Rmk. 2.3, Sect. 2.4]{Beauville}\footnote{Note that some of these are phrased in the language of coadjoint orbits, rather than adjoint orbits. However, recall from \ref{Subsection: Geometric features} that one can identify coadjoint orbits with adjoint orbits.}), so that $\mathbb{P}(\mathcal{O})$ is a complex contact manifold. In particular, $\mathbb{P}(\mathcal{O}_{\text{min}})$ belongs to the class of Fano contact manifolds (see \cite[Def. 1.3, Def. 2.5]{LeBrunSalamon}), which are of interest to quaternionic-K\"{a}hler geometers (see \cite[Sect. 1]{LeBrunSalamon}). Moreover, it is in the quaternionic-K\"{a}hler context that LeBrun and Salamon conjectured the following classification of Fano contact manifolds: each is isomorphic, as a complex contact manifold, to the projectivization of the minimal nilpotent orbit of a simple algebraic group. This conjecture is known to be true in a number of cases (see \cite{Beauville}), and counter-examples have not yet been found.

\begin{remark}
When $\mathbb{P}(\mathcal{O}_{\text{min}})$ is described in classical terms, as in Example \ref{Example: Description of P(Omin) in type A}, one should expect a similarly classical description of its complex contact structure. Such a description is given in \cite[Sect. 4.2]{CrooksRayan} in Lie type $D_n$, when $\mathbb{P}(\mathcal{O}_{\text{min}})$ is viewed as the Grassmannian of isotropic $2$-dimensional subspaces of $\mathbb{C}^{2n}$.
\end{remark} 

\subsection{Orbit closures and singularities}\label{Subsection: Orbit closures and singularities}
It turns out that nilpotent orbit closures have been studied extensively in the literature, both intrinsically and as parts of broader research programs. We now briefly survey some of this literature, giving references where necessary in order to avoid long digressions. 

Recall that the closure of $\mathcal{O}_{\text{reg}}$ is the nilpotent cone $\mathcal{N}$. This is singular variety and there exists a canonical resolution of singularities $\mu:T^*(G/B)\rightarrow\mathcal{N}$ \footnote{By ``resolution of singularities'', we mean that $\mu$ is a proper surjective morphism with the property of being an isomorphism over the smooth locus of $\mathcal{N}$.}, called the \textit{Springer resolution} (see \cite[Def. 3.2.4]{Chriss}\footnote{Note that this reference substitutes $G/B$ with the variety of all Borel subalgebras of $\mathfrak{g}$.}). While fundamental to the study of singularities in $\mathcal{N}$, $\mu$ also plays a significant role in adjacent subjects. In geometric representation theory, one constructs a representation of the Weyl group on the Borel-Moore homology of each fibre of $\mu$ (see \cite[Sect. 3.4--3.6]{Chriss}). In symplectic geometry, $T^*(G/B)$ and $\mathcal{O}_{\text{reg}}$ have canonical symplectic forms (see \cite[Sect. 1.1]{Chriss} for the former and \ref{Subsection: Geometric features} for the latter) and $\mu$ is an example of a \textit{symplectic resolution} (see \cite[Sect. 1.3]{Fu} for a definition, and all of \cite{Fu} for an introduction to the subject). This turns out to be an instance of a more general occurrence: for certain nilpotent orbits $\mathcal{O}\subseteq\mathfrak{g}$, one can find a parabolic subgroup $P\subseteq G$ and a symplectic resolution $T^*(G/P)\rightarrow\overline{\mathcal{O}}$ (see \cite[Sect. 5.1]{Fu}\footnote{This reference focuses on resolutions of normalizations of orbit closures, but many closures happen to be normal (see \cite{Kraft}, for example).}). These nilpotent orbit closures thereby give examples of \textit{symplectic singularities} (see \cite[Def. 1.1]{BeauvilleSingular}), which have received a great deal of recent attention (see \cite{BeauvilleSingular,FuNamikawa,Ginzburg,Kaledin}).            

In addition to the above, nilpotent orbits closures are sometimes studied through their intersections with certain affine-linear subsets of $\mathfrak{g}$, often called \textit{Slodowy slices} (see \cite[Prop. 3.7.15]{Chriss}). This study was initiated by the works of Brieskorn \cite{Brieskorn} and Slodowy \cite{Slodowy}, which show the intersection of $\overline{\mathcal{O}_{\text{reg}}}=\mathcal{N}$ with a particular Slodowy slice to contain a Kleinian surface singularity when $\mathfrak{g}$ is of type $ADE$ (see \cite[Sect. 6.4]{Slodowy} for further details). This result is interesting in many contexts, and in particular plays a role in Kronheimer's study of nilpotent orbits in hyperk\"{a}hler geometry (see \cite[Prop. 2(a), Prop. 2(b)]{KronheimerInstantons}).     
\bibliographystyle{acm} 
\bibliography{Adjoint}

\def\cprime{$'$}
\begin{thebibliography}{10}

\bibitem{Akhiezer}
{\sc Akhiezer, D.~N.}
\newblock {\em Lie group actions in complex analysis}.
\newblock Aspects of Mathematics, E27. Friedr. Vieweg \& Sohn, Braunschweig,
  1995.

\bibitem{Azad}
{\sc Azad, H., van~den Ban, E., and Biswas, I.}
\newblock Symplectic geometry of semisimple orbits.
\newblock {\em Indag. Math. (N.S.) 19}, 4 (2008), 507--533.

\bibitem{Ballico}
{\sc Ballico, E., Gasparim, E., Grama, L., and San~Martin, L.}
\newblock Some {L}andau-{G}inzburg models viewed as rational maps.
  arxiv:1601.05119 (2016), 16 pages.

\bibitem{Vogan}
{\sc Barbasch, D., and Vogan, Jr., D.~A.}
\newblock Unipotent representations of complex semisimple groups.
\newblock {\em Ann. of Math. (2) 121}, 1 (1985), 41--110.

\bibitem{Beauville}
{\sc Beauville, A.}
\newblock Fano contact manifolds and nilpotent orbits.
\newblock {\em Comment. Math. Helv. 73}, 4 (1998), 566--583.

\bibitem{BeauvilleSingular}
{\sc Beauville, A.}
\newblock Symplectic singularities.
\newblock {\em Invent. Math. 139}, 3 (2000), 541--549.

\bibitem{Bernstein}
{\sc Bernstein, I.~N., Gel'fand, I.~M., and Gel'fand, S.~I.}
\newblock Schubert cells, and the cohomology of the spaces {$G/P$}.
\newblock {\em Uspehi Mat. Nauk 28}, 3(171) (1973), 3--26.

\bibitem{BielawskiOnthe}
{\sc Bielawski, R.}
\newblock On the hyperk\"ahler metrics associated to singularities of nilpotent
  varieties.
\newblock {\em Ann. Global Anal. Geom. 14}, 2 (1996), 177--191.

\bibitem{Bielawski}
{\sc Bielawski, R.}
\newblock Lie groups, {N}ahm's equations and hyperk\"ahler manifolds.
\newblock In {\em Algebraic groups}. Universit\"atsverlag G\"ottingen,
  G\"ottingen, 2007, pp.~1--17.

\bibitem{BielawskiReducible}
{\sc Bielawski, R.}
\newblock Reducible spectral curves and the hyperk\"ahler geometry of adjoint
  orbits.
\newblock {\em J. Lond. Math. Soc. (2) 76}, 3 (2007), 719--738.

\bibitem{Biquard}
{\sc Biquard, O.}
\newblock Sur les \'equations de {N}\"{a}hm et la structure de {P}oisson des
  alg\`ebres de {L}ie semi-simples complexes.
\newblock {\em Math. Ann. 304}, 2 (1996), 253--276.

\bibitem{Borho}
{\sc Borho, W., and MacPherson, R.}
\newblock Repr\'esentations des groupes de {W}eyl et homologie d'intersection
  pour les vari\'et\'es nilpotentes.
\newblock {\em C. R. Acad. Sci. Paris S\'er. I Math. 292}, 15 (1981), 707--710.

\bibitem{Brieskorn}
{\sc Brieskorn, E.}
\newblock Singular elements of semi-simple algebraic groups.
\newblock In {\em Actes du {C}ongr\`es {I}nternational des {M}ath\'ematiciens
  ({N}ice, 1970), {T}ome 2}. Gauthier-Villars, Paris, 1971, pp.~279--284.

\bibitem{Brion}
{\sc Brion, M.}
\newblock Equivariant cohomology and equivariant intersection theory.
\newblock In {\em Representation theories and algebraic geometry ({M}ontreal,
  {PQ}, 1997)}, vol.~514 of {\em NATO Adv. Sci. Inst. Ser. C Math. Phys. Sci.}
  Kluwer Acad. Publ., Dordrecht, 1998, pp.~1--37.
\newblock Notes by Alvaro Rittatore.

\bibitem{Bump}
{\sc Bump, D.}
\newblock {\em Lie groups}, vol.~225 of {\em Graduate Texts in Mathematics}.
\newblock Springer-Verlag, New York, 2004.

\bibitem{Chriss}
{\sc Chriss, N., and Ginzburg, V.}
\newblock {\em Representation theory and complex geometry}.
\newblock Modern Birkh\"auser Classics. Birkh\"auser Boston, Inc., Boston, MA,
  2010.
\newblock Reprint of the 1997 edition.

\bibitem{Collingwood}
{\sc Collingwood, D.~H., and McGovern, W.~M.}
\newblock {\em Nilpotent orbits in semisimple {L}ie algebras}.
\newblock Van Nostrand Reinhold Mathematics Series. Van Nostrand Reinhold Co.,
  New York, 1993.

\bibitem{CrooksRayan}
{\sc Crooks, P., and Rayan, S.}
\newblock Some results on equivariant contact geometry for partial flag
  varieties.
\newblock {\em Internat. J. Math. 27}, 8 (2016), 1650066, 13pp.

\bibitem{Santa-Cruz}
{\sc D'Amorim Santa-Cruz, S.}
\newblock Twistor geometry for hyper-{K}\"ahler metrics on complex adjoint
  orbits.
\newblock {\em Ann. Global Anal. Geom. 15}, 4 (1997), 361--377.

\bibitem{Dancer}
{\sc Dancer, A.~S.}
\newblock Hyperk\"ahler manifolds.
\newblock In {\em Surveys in differential geometry: essays on {E}instein
  manifolds}, vol.~6 of {\em Surv. Differ. Geom.} Int. Press, Boston, MA, 1999,
  pp.~15--38.

\bibitem{Digne}
{\sc Digne, F., and Michel, J.}
\newblock {\em Representations of finite groups of {L}ie type}, vol.~21 of {\em
  London Mathematical Society Student Texts}.
\newblock Cambridge University Press, Cambridge, 1991.

\bibitem{Fritzsche}
{\sc Fritzsche, K., and Grauert, H.}
\newblock {\em From holomorphic functions to complex manifolds}, vol.~213 of
  {\em Graduate Texts in Mathematics}.
\newblock Springer-Verlag, New York, 2002.

\bibitem{Fu}
{\sc Fu, B.}
\newblock A survey on symplectic singularities and symplectic resolutions.
\newblock {\em Ann. Math. Blaise Pascal 13}, 2 (2006), 209--236.

\bibitem{FuNamikawa}
{\sc Fu, B., and Namikawa, Y.}
\newblock Uniqueness of crepant resolutions and symplectic singularities.
\newblock {\em Ann. Inst. Fourier (Grenoble) 54}, 1 (2004), 1--19.

\bibitem{Fulton}
{\sc Fulton, W.}
\newblock {\em Young tableaux}, vol.~35 of {\em London Mathematical Society
  Student Texts}.
\newblock Cambridge University Press, Cambridge, 1997.
\newblock With applications to representation theory and geometry.

\bibitem{GGSM}
{\sc Gasparim, E., Grama, L., and San~Martin, L. A.~B.}
\newblock Symplectic {L}efschetz fibrations on adjoint orbits.
\newblock {\em Forum Math. 28}, 5 (2016), 967--979.

\bibitem{Gerstenhaber}
{\sc Gerstenhaber, M.}
\newblock On dominance and varieties of commuting matrices.
\newblock {\em Ann. of Math. (2) 73\/} (1961), 324--348.

\bibitem{Ginzburg}
{\sc Ginzburg, V., and Kaledin, D.}
\newblock Poisson deformations of symplectic quotient singularities.
\newblock {\em Adv. Math. 186}, 1 (2004), 1--57.

\bibitem{Givental}
{\sc Givental, A.~B.}
\newblock Homological geometry and mirror symmetry.
\newblock In {\em Proceedings of the {I}nternational {C}ongress of
  {M}athematicians, {V}ol.\ 1, 2 ({Z}\"urich, 1994)\/} (1995), Birkh\"auser,
  Basel, pp.~472--480.

\bibitem{Graham}
{\sc Graham, W.}
\newblock Positivity in equivariant {S}chubert calculus.
\newblock {\em Duke Math. J. 109}, 3 (2001), 599--614.

\bibitem{Graham2}
{\sc Graham, W.~A.}
\newblock Functions on the universal cover of the principal nilpotent orbit.
\newblock {\em Invent. Math. 108}, 1 (1992), 15--27.

\bibitem{Guillemin}
{\sc Guillemin, V., Holm, T., and Zara, C.}
\newblock A {GKM} description of the equivariant cohomology ring of a
  homogeneous space.
\newblock {\em J. Algebraic Combin. 23}, 1 (2006), 21--41.

\bibitem{Hesselink}
{\sc Hesselink, W.}
\newblock Singularities in the nilpotent scheme of a classical group.
\newblock {\em Trans. Amer. Math. Soc. 222\/} (1976), 1--32.

\bibitem{Hochschild}
{\sc Hochschild, G.~P.}
\newblock {\em Basic theory of algebraic groups and {L}ie algebras}, vol.~75 of
  {\em Graduate Texts in Mathematics}.
\newblock Springer-Verlag, New York-Berlin, 1981.

\bibitem{Hori}
{\sc Hori, K., and Vafa, C.}
\newblock Mirror {S}ymmetry. arxiv:hep-th/0002222 (2000), 91 pages.

\bibitem{Humphreys}
{\sc Humphreys, J.~E.}
\newblock {\em Linear algebraic groups}.
\newblock Springer-Verlag, New York-Heidelberg, 1975.
\newblock Graduate Texts in Mathematics, No. 21.

\bibitem{HumphreysConjugacy}
{\sc Humphreys, J.~E.}
\newblock {\em Conjugacy classes in semisimple algebraic groups}, vol.~43 of
  {\em Mathematical Surveys and Monographs}.
\newblock American Mathematical Society, Providence, RI, 1995.

\bibitem{Jantzen}
{\sc Jantzen, J.~C.}
\newblock Nilpotent orbits in representation theory.
\newblock In {\em Lie theory}, vol.~228 of {\em Progr. Math.} Birkh\"auser
  Boston, Boston, MA, 2004, pp.~1--211.

\bibitem{Kaledin}
{\sc Kaledin, D.}
\newblock Symplectic singularities from the {P}oisson point of view.
\newblock {\em J. Reine Angew. Math. 600\/} (2006), 135--156.

\bibitem{Knapp}
{\sc Knapp, A.~W.}
\newblock {\em Lie groups beyond an introduction}, second~ed., vol.~140 of {\em
  Progress in Mathematics}.
\newblock Birkh\"auser Boston, Inc., Boston, MA, 2002.

\bibitem{Kobak}
{\sc Kobak, P.~Z., and Swann, A.}
\newblock Classical nilpotent orbits as hyperk\"ahler quotients.
\newblock {\em Internat. J. Math. 7}, 2 (1996), 193--210.

\bibitem{Kostant}
{\sc Kostant, B.}
\newblock The principal three-dimensional subgroup and the {B}etti numbers of a
  complex simple {L}ie group.
\newblock {\em Amer. J. Math. 81\/} (1959), 973--1032.

\bibitem{KostantPoly}
{\sc Kostant, B.}
\newblock Lie group representations on polynomial rings.
\newblock {\em Amer. J. Math. 85\/} (1963), 327--404.

\bibitem{Kovalev}
{\sc Kovalev, A.~G.}
\newblock Nahm's equations and complex adjoint orbits.
\newblock {\em Quart. J. Math. Oxford Ser. (2) 47}, 185 (1996), 41--58.

\bibitem{Kraft}
{\sc Kraft, H., and Procesi, C.}
\newblock Closures of conjugacy classes of matrices are normal.
\newblock {\em Invent. Math. 53}, 3 (1979), 227--247.

\bibitem{KronheimerSemisimple}
{\sc Kronheimer, P.~B.}
\newblock A hyper-{K}\"ahlerian structure on coadjoint orbits of a semisimple
  complex group.
\newblock {\em J. London Math. Soc. (2) 42}, 2 (1990), 193--208.

\bibitem{KronheimerInstantons}
{\sc Kronheimer, P.~B.}
\newblock Instantons and the geometry of the nilpotent variety.
\newblock {\em J. Differential Geom. 32}, 2 (1990), 473--490.

\bibitem{LeBrunSalamon}
{\sc LeBrun, C., and Salamon, S.}
\newblock Strong rigidity of positive quaternion-{K}\"ahler manifolds.
\newblock {\em Invent. Math. 118}, 1 (1994), 109--132.

\bibitem{Atiyah}
{\sc Luke, G.~L.}, Ed.
\newblock {\em Representation theory of {L}ie groups\/} (1979), vol.~34 of {\em
  London Mathematical Society Lecture Note Series}, Cambridge University Press,
  Cambridge-New York.

\bibitem{Mihaleca}
{\sc Mihalcea, L.~C.}
\newblock On equivariant quantum cohomology of homogeneous spaces: {C}hevalley
  formulae and algorithms.
\newblock {\em Duke Math. J. 140}, 2 (2007), 321--350.

\bibitem{Schmitt}
{\sc Schmitt, A. H.~W.}
\newblock {\em Geometric invariant theory and decorated principal bundles}.
\newblock Zurich Lectures in Advanced Mathematics. European Mathematical
  Society (EMS), Z\"urich, 2008.

\bibitem{Slodowy}
{\sc Slodowy, P.}
\newblock {\em Simple singularities and simple algebraic groups}, vol.~815 of
  {\em Lecture Notes in Mathematics}.
\newblock Springer, Berlin, 1980.

\bibitem{Sommers}
{\sc Sommers, E.}
\newblock A generalization of the {B}ala-{C}arter theorem for nilpotent orbits.
\newblock {\em Internat. Math. Res. Notices}, 11 (1998), 539--562.

\bibitem{Sommers2}
{\sc Sommers, E.}
\newblock Lusztig's canonical quotient and generalized duality.
\newblock {\em J. Algebra 243}, 2 (2001), 790--812.

\bibitem{Springer}
{\sc Springer, T.~A.}
\newblock A construction of representations of {W}eyl groups.
\newblock {\em Invent. Math. 44}, 3 (1978), 279--293.

\bibitem{Swann}
{\sc Swann, A.}
\newblock Hyperk\"ahler and quaternionic {K}\"ahler geometry.
\newblock {\em Math. Ann. 289}, 3 (1991), 421--450.

\bibitem{Tauvel}
{\sc Tauvel, P., and Yu, R. W.~T.}
\newblock {\em Lie algebras and algebraic groups}.
\newblock Springer Monographs in Mathematics. Springer-Verlag, Berlin, 2005.

\bibitem{Villumsen}
{\sc Villumsen, M.}
\newblock Cohomogeneity-three hyper{K}\"ahler metrics on nilpotent orbits.
\newblock {\em Ann. Global Anal. Geom. 28}, 2 (2005), 123--156.

\bibitem{Whitehead}
{\sc Whitehead, G.~W.}
\newblock {\em Elements of homotopy theory}, vol.~61 of {\em Graduate Texts in
  Mathematics}.
\newblock Springer-Verlag, New York-Berlin, 1978.

\end{thebibliography}
\end{document}